\documentclass[10pt]{amsart}
\usepackage{amsmath, latexsym, 
}
\usepackage{amssymb,amsthm}
\usepackage{amscd,graphics}

\newtheorem{proposition}{Proposition}[section]
\newtheorem{lemma}[proposition]{Lemma}

\newtheorem{theorem}[proposition]{Theorem}

\newtheorem{corollary}[proposition]{Corollary}

\theoremstyle{remark}

\newtheorem{remark}[proposition]{Remark}

\theoremstyle{definition}
\newtheorem*{definition}{Definition}
\def\real{\mathbb{R}}
\def\integer{\mathbb{Z}}
\def\complex{\mathbb{C}}
\def\supp{\mathrm{supp}}
\def\var{\mathrm{var}}

\def\sp{\mathrm{sp}}
\def\id{\mathrm{id}}

\def\BB{\mathcal{B}}
\def\DD{\mathcal{D}}
\def\EE{\mathcal{E}}
\def\FF{\mathcal{F}}
\def\GG{\mathcal{G}}
\def\II{\mathcal{I}}
\def\JJ{\mathcal{J}}
\def\KK{\mathcal{K}}
\def\LL{\mathcal{L}}
\def\MM{\mathcal{M}}
\def\NN{\mathcal{N}}

\def \PP{\mathcal {P}}
\def \QQ{\mathcal {Q}}
\def \RR{\mathcal {R}}
\def\SS{\mathcal{S}}

\begin{document}
\title{Linear response formula for piecewise expanding unimodal maps}
\author{Viviane Baladi and Daniel Smania} 
\address{UMI 2924 CNRS-IMPA,  Estrada Dona Castorina 110,
22460-320 Rio de Janeiro, Brazil;
and CNRS, UMR 7586, Institut de Math\'ematique de 
Jussieu, Paris;
Current address: D.M.A., UMR 8553,\'Ecole Normale Sup\'erieure,  75005 Paris, France}
\email{viviane.baladi@ens.fr}

\address{
Departamento de Matem\'atica,
ICMC-USP, Caixa Postal 668,  S\~ao Carlos-SP,
CEP 13560-970
S\~ao Carlos-SP, Brazil}
\email{smania@icmc.usp.br}
\date{\today} 
\begin{abstract}
The average $\RR(t)=\int \varphi\, d\mu_t$
of a smooth function $\varphi$
with respect to the SRB measure $\mu_t$ of a smooth one-parameter family
$f_t$ of piecewise expanding interval maps
is not always Lipschitz  \cite{Ba}, \cite{MM}.
We prove that if $f_t$ is tangent to the topological class
of $f$, and if $\partial_t f_t|_{t=0}=X\circ f$, then $\RR(t)$ is differentiable at zero,
and  $\RR'(0)$ coincides with the resummation
proposed in \cite{Ba} of the (a priori divergent) series 
$\sum_{n=0}^\infty
\int X(y) \partial_y (\varphi \circ f^n)(y) \, d\mu_0(y)$
given by Ruelle's conjecture.
In fact, we show that $t \mapsto \mu_t$ is differentiable within
Radon measures. Linear response is violated 
if and only if $f_t$ is transversal to the topological class
of $f$.
\end{abstract}
\thanks{V.B. is partially supported by ANR-05-JCJC-0107-01.
D.S. is partially supported by CNPq 470957/2006-9
and 310964/2006-7, FAPESP 2003/03107-9. Both authors
thank the organisers of the workshop Topologia e Din\^amica
at UFF, Niteroi, Brasil, February 2007, where this work was started.
V.B. gratefully acknowledges the warm hospitality
of Universidad de la Rep\'ublica, Montevideo, Uruguay, where
part of this work was done.}
\maketitle


\section{Introduction}

Let us call SRB measure
for  a dynamical system $f: \MM \to \MM$, on a 
manifold $\MM$ endowed with Lebesgue measure, 
an $f$-invariant ergodic probability measure $\mu$ so that  the set
$
\{ x \in \MM\mid \lim_{n \to \infty}
\frac{1}{n} \sum_{k=0}^{n-1} \varphi(f^k(x))= \int \varphi \, d\mu \} 
$
has positive Lebesgue measure, for continuous
observables $\varphi$.
(In fact this defines a {\it physical measure,}  see
e.g. \cite{Yo}.)
If $f_t$ is a smooth one-parameter family
with $f_0=f$, and each $f_t$
admits a unique SRB measure $\mu_t$, it is natural to ask
how $\mu_t$ depends on $t$. 
More precisely, one studies, for fixed smooth enough $\varphi$, the
function
$
\RR(t)= \int \varphi \, d\mu_t
$.

If $f$ is a sufficiently
smooth uniformly hyperbolic diffeomorphism restricted to a transitive attractor,
Ruelle \cite{Ruok}--\cite{Ruok'} proved
 that $\RR(t)$
is differentiable 
at $t=0$.
In addition, Ruelle gave an explicit formula for
$\RR'(0)$, depending on $f_t$ only through its linear
part (the ``infinitesimal
deformation") $v= \partial_t f_t|_{t=0}$. For obvious reasons, this formula is called the
{\it linear response formula.}
See  \cite[Cor. 1 p. 595]{KKPW} -- noting that
$f$ and $\rho$ in the statement there need in fact only be H\" older  --
for a previous results in continuous-time the Anosov setting, without
an explicit formula for $\RR'(0)$.
We refer to the introductions of \cite{Dolgo}, \cite{BuL},  \cite{Ba}, for a discussion of more
references regarding linear response for hyperbolic
dynamical systems, including \cite{BuL}, \cite{BFG}, \cite{JL}, and applications
to statistical mechanics \cite{Gal}.

A much more difficult situation consists in studying
nonuniformly hyperbolic interval maps $f$, e.g. smooth
unimodal maps.  For some of these maps, in particular those which
satisfy the Collet-Eckmann condition, there exists a unique
SRB measure $\mu$. Two new difficulties are
that structural stability does not hold (in a rather drastic
way\footnote{As was explained to us by D. Dolgopyat,
the examples in \cite[Section 2.3(B)]{Dolgo} may fail to be structurally
stable. However, shadowing holds for a sufficiently
large measure of points so that 
Theorem~ 1--Proposition ~ 2.6  of \cite{Dolgo} provide a linear
response formula in the sense of Whitney.}), and that $f_t$ will not
always have an SRB measure even if $f$ has one.
In this setting,  Ruelle (\cite{Rupr}, \cite{Ru3})
has outlined a program,  for infinitesimal
deformations of the form $v=X \circ f$. He proposed
$\Psi(1)$, where 
\begin{equation}\label{intro}
\Psi(z)=\sum_{n=0}^\infty
\int z^n X(y)   {\partial_y} (\varphi \circ f^n) (y)\, d\mu_0(y)\, ,
\end{equation}
is the ``susceptibility function,"
\footnote{Since $\Psi(e^{i\omega})$
is the Fourier transform of the ``linear response" 
\cite{Ru1ph}, it is natural to consider the variable
$\omega$, but we prefer to work with the variable $z=e^{i\omega}$.}
as a candidate 
for the derivative,  in the sense of Whitney's extension, of $\RR(t)$ at $t=0$. 
(We refer
e.g. to the introduction of \cite{Ba} for more details.)
Beware that the series (\ref{intro})
may diverge at $z=1$ so that  $\Psi(1)$ needs to be suitably
interpreted.

\smallskip
In this paper, just like in \cite{Ba}, we consider a simpler situation
which exhibits however a similar bifurcation structure
(in particular structural stability does not hold and infinitely
many symbols may be required to code the dynamics):
piecewise expanding interval maps. For such maps, it has been known
for some time that $\mu_t$ exists for all  $t$, and, 
under mild assumptions, that
$\RR(t)$  has modulus of continuity
$O(t \ln|t|)$ (see (\ref{claim'}) below and the references given there).
We view the setting of piecewise expanding interval maps as
a laboratory in which to test our ideas about smooth deformations.
The arguments are free from technicalities, but 
exhibit most of the features that will appear in the
Collet-Eckmann case.

Let us recall now recent results in this piecewise expanding setting.
Assuming  that $\partial_t f_t|_{t=0}=X\circ f$, a 
function $(f,X)\mapsto \JJ(f,X)$ was introduced in \cite{Ba} 
(see (\ref{saltusid})).
There exist (\cite{Ba}, \cite{MM}) examples
of piecewise expanding unimodal interval maps $f_t$  so that  
$\RR(t)$ is not Lipschitz.
For these counterexamples, it turns out that $\JJ(f,X) \ne 0$.
The function $\Psi(z)$ is holomorphic \cite{Ba}  in the
open unit disc. 
In addition, if  $\JJ(f,X)=0$ and $f$ is Markov
(i.e., the postcritical orbit  is finite) then
$\Psi(z)$ is holomorphic at $z=1$ (\cite{Ba}). If $\JJ(f,X)=0$ but $f$ is
not Markov  a resummation $\Psi_1$ was devised \cite{Ba} for
the possibly divergent series $\Psi(1)$
(see Proposition~\ref{candidate} below).
In view of the above facts (see also \cite[Remark 4.5]{Ba}), 
a modification  
of Ruelle's conjecture, 
was proposed in \cite[Conjecture A]{Ba} for perturbations
of piecewise
expanding or Collet-Eckmann $f$, assuming in addition
that each $f_t$ {\it is topologically conjugated to}
~$f$. 

\smallskip
The main result of this paper is the proof of Conjecture A from \cite{Ba} in the
piecewise expanding setting. In fact, we prove a slightly stronger
result (Theorem~\ref{formula}): It is enough to assume that 
$f_t$ is {\it tangent} to the topological class of $f$ (see \S2.1).
Also, the observable $\varphi$ need only be continuous, so that
in fact we prove that $t \mapsto \mu_t$ is differentiable into Radon measures.
The interpretation of $\Psi(1)$ in Theorem~\ref{formula} is
in the sense of $\Psi_1$ from  \cite{Ba}, and we find a more
compact expression for $\Psi_1$, as well as a condition ensuring that
$\Psi_1$ is the abelian limit of $\Psi(z)$ (Proposition~ \ref{abelian}).

Our approach to
prove Theorem~\ref{formula}
is a perturbative spectral analysis (via resolvents)  of transfer operators,
on suitable spaces, adapted from those in \cite{Ba}.
\footnote{The spaces in \cite{Ba} were inspired by what Ruelle told us about his then
ongoing work on the nonuniformly expanding case \cite{Ruwip}.}
(In spirit, this is somewhat similar to the work of Butterley-Liverani \cite{BuL}.)
To perform this analysis, we  use the Keller-Liverani
\cite{KL} results together with
smooth motions
(Proposition~\ref{TCE}) and the twisted cohomological
equation for $f$ and $X \circ f$. The novelty of this
work resides in the combination of these two ingredients.
A key new ingredient in the implementation of our ideas
is the use of the isometry $G_t$ in the proof of Theorem~\ref{formula}:
this isometry is the device which allows us to use the same Banach space
for the transfer operators of all perturbations, by forcing the
singularities (here, jumps) to lie on a prescribed set.

\smallskip
We next summarise informally the picture for  piecewise expanding,
piecewise smooth unimodal maps (see \S~\ref{2.1} for assumptions).
If the critical point is not periodic, noting $f^0=\id$,
we say that
$v$ is horizontal for $f$
if $ \sum_{j = 0}^{\infty} \frac{v(f^j(c))}{(f^{j})'(f(c))}=0$
(see (\ref{Mf}) for the periodic case). Then:

 \renewcommand{\labelenumi}{(\roman{enumi})}
\begin{enumerate}
\item
 $\JJ(f,X)=0$
if and only if  $X$ is
horizontal for $f$ (Corollary \ref{horiz}).
\item
$X\circ f$ is  horizontal for $f$  if and only
if the candidate $\Psi_1$  from \cite{Ba} for the derivative is well-defined
(Proposition~\ref{candidate} from \cite{Ba},  Proposition~\ref{nono}).
\item
If $f_t$ is tangent to the topological class of $f$ then
$\partial_t f_t|_{t=0}$ is horizontal for $f$
(Corollary~\ref{horiz}).
\item
If $v$  is horizontal for $f$, then any 
$f_t$  with $\partial_t  f_t|_{t=0}=v$
is tangent to the topological class of $f$. (Theorem~\ref{???} below, to appear in \cite{??}.)
\item
If $f_t$ is  stably mixing
\footnote{Beware that if $f$ is not stably mixing, then
there exist $f_t$ with $\partial_t f_t|_{t=0}=X\circ f$
horizontal and $\Psi(z)$
holomorphic at $0$, but $\RR(t)$ not Lipschitz.} and tangent to the topological class of $f$ 
with $\partial_t f_t|_{t=0}=X\circ f$,
then $\RR(t)$ is differentiable at $t=0$, and the linear response
formula $\RR'(0)=\Psi_1$ holds
(Theorem~\ref{formula}).
\item
If
$\partial_t f_t|_{t=0}$ is  not horizontal and $c$ is not periodic
for $f$ then  there exists $C^\infty$ observable $\varphi$
so that $\RR(t)$ is not Lipschitz 
(Theorem~\ref{nonlip},
see \cite{Ba}, \cite{MM} for isolated examples).
\end{enumerate}

\medskip

In view of the results of the present paper, we expect that the following
strengthening of Conjecture A \cite{Ba} in the Collet-Eckmann case  holds:

\smallskip
{\bf Conjecture $\mathbf {A'}$.}
Let $f$ be   a mixing smooth Collet-Eckmann unimodal map with
a nonflat critical point.
Let $f_t$ be a smooth perturbation,
with $f_0=f$ and $\partial_t f_t|_{t=0}=X \circ f$, which is
{\it tangent to the topological class of $f$} (i.e., so that there
exists $\tilde f_t$
such that $|\tilde f_t-f_t|=O(t^2)$
and each $\tilde f_t$ is topologically conjugated to
$f$). Then $\RR(t)$ is differentiable at $0$
{\it in the sense of Whitney} for
all smooth observables $\varphi$, and $\RR'(0)= \Psi(1)$ (the
infinite sum being suitably interpreted).

\smallskip
In particular, if $f_t$ remains in the topological class of a
Collet-Eckmann map $f$, Conjecture~A' is just \cite[Conjecture A]{Ba}, where
differentiability of $\RR(t)$ is foreseen in the usual sense.
We  expect (see Conjecture B in \cite{Ba}) that  paths $f_t$
which are {\it not}  tangent to conjugacy classes
give rise to $\RR(t)$
which are in general  H\"older but not Lipschitz in the sense of Whitney.
Note that topological classes are called hybrid classes
in this context, and they form a  well understood lamination for 
smooth  maps with a quadratic critical point 
(see \cite{Ly}, \cite{ALM} and references therein).

\medskip

This work is about the linear response. One can
also wonder about formulas for the derivatives of higher order
of $\RR(t)$ (see \cite{Ruhi}).
Indeed, we expect that a suitable modification
of the proof of Theorem~\ref{formula} will give, if
$f_t$ is a $C^{r_0, r_0+1}$
perturbation, tangent up to order $r_0-1$ to the topological class
of a stably mixing piecewise expanding unimodal map $f$
(i.e., we replace $|f_t-\tilde f_t|=0(t^2)$ by
$O(t^{r_0})$ for $r_0 \ge 3$ in \S~\ref{2.1}), that $\RR(t)$ has
a Taylor series of degree $r_0-1$ at $0$, with explicit
coefficients  (in the spirit of \cite{Ruhi}). The coefficients
will be related to twisted cohomological equations for derivatives
of higher order of $h_t$ (see the proof of Proposition~\ref{TCE}).
In the Collet-Eckmann setting, if
$f_t$ is tangent to the hybrid class of $f$
up to order $r_0-1$, then we expect that higher order derivatives
and Taylor series of degree $r_0-1$ should be attainable,
of course in the sense of Whitney. (If $f_t$ lies in the hybrid class,
we expect a Taylor series in the usual calculus sense.)

\smallskip

The paper is organised as follows: Section~\ref{two} contains
definitions, and the essential result on the ``smooth motions" $h_t(x)$
(Proposition~\ref{TCE}). The
infinitesimal conjugacy $\alpha=\partial_t h_t|_{t=0}$
is introduced there. In Section~\ref{three}, we recall
the decomposition of the invariant density from
\cite{Ba}, we  adapt results from \cite{KL} on the transfer operators to
reduce from families tangent to the topological class
to families within the topological class (Proposition~\ref{fromKL}),
and we introduce appropriate spaces $\BB_{t}$ for transfer operators (Subsection~\ref{3.2})
of sums of a ``smooth" function with a sum of jumps along the postcritical orbit.
In Section ~\ref{fourr}, we recall information from \cite{Ba}
on the susceptibility function $\Psi(z)$ and the candidate
$\Psi_1$  for the derivative of $\RR(t)$.
We prove Theorem~\ref{formula} in Section~ ~\ref{four}, combining
the main ingredients (Proposition~\ref{TCE}, Proposition~\ref{fromKL},
and the spectral analysis on the function spaces $\BB_{t}$ from Subsection~\ref{3.2}).
The proof uses strongly the perturbation
theory  from Keller and Liverani \cite{KL} (we need to extend
their result slightly, see Appendix~\ref{newKL}).
Section \ref{elegg} contains (Theorem~\ref{elegant})
a simpler formula for $\RR'(0)$, which is true if and only
if $\alpha$ is absolutely continuous (a rare event).
Theorem~\ref{nonlip} in the last section shows that the condition to be
tangent to the topological class is necessary. 

\bigskip
After the first version of the present paper was made public,
David Ruelle sent us a copy of \cite{RuStr},
which contains in particular
a proof of \cite[Conjecture~ A]{Ba} under the additional assumptions
that $f_0$ is analytic and has a nonrecurrent postcritical
orbit.
We hope that injecting
in our argument  tools analogous to those  developed there
should eventually give a proof of Conjecture ~A' for Collet--Eckmann maps.

\section{The setting, the twisted cohomological equation and the
infinitesimal conjugacy $\alpha$} 
\label{two}

\subsection{Piecewise expanding $C^{r}$ unimodal maps and their perturbations}
\label{2.1}

If $K\subset \real$ is a compact interval and $\ell \ge 0$,
we let $C^\ell(K)$ denote the set of functions on $K$
which extend to $C^\ell$ functions in an open neighbourhood of $K$.
In this work, we consider the following objects:

\begin{definition} 
For an integer $r \ge 1$, a {\it piecewise expanding $C^{r}$ unimodal map}
is a continuous map
$f:I\to I$, where $I=[a,b]$,  so that $f$ is strictly increasing
on $I_+=[a,c]$, strictly decreasing on $I_-=[c,b]$
($a<c<b$), with $f(a)=f(b)=a$;
and
for $\sigma=\pm$, the map
$f|_{I_\sigma}$ extends to a  $C^{r}$ map
on a neighbourhood of $I_\sigma$, with
\footnote{A prime denotes derivation with respect to $x\in I$, a priori in the sense of distributions.} $\inf |f'|_{I_\sigma}|  > 1$.

\noindent
A piecewise expanding $C^{r}$ unimodal map $f$ is {\it good} if
either $c$ is  not periodic under $f$
or  $\inf |(f^{n_1})' |> 2$, where $n_1\ge 2$ is the minimal
period of $c$; it is  {\it mixing} if $f$ is topologically mixing on
$[f^2(c), f(c)]$.
\end{definition}

Beware that  a  piecewise expanding $C^{r}$ unimodal map $f$
is only continuous, and never $C^1$ (it is piecewise $C^r$).
We restrict to unimodal (as opposed to multimodal)
to avoid unessential combinatorical difficulties.

Given  a  piecewise expanding $C^{r}$ unimodal map $f$, we shall use the following
notation:
The point $c$ will be called the {\it critical point} of  $f$.
We write $c_k=f^k(c)$ for $k \ge 0$.  
We say that $c$ is {\it preperiodic}
if it is not periodic but there exist $n_0\ge 1$ and $n_1\ge 1$
so that $c_{n_0}$ is periodic of minimal period $n_1$
(we take $n_0$ minimal for this property and our
assumptions imply $n_0\ge 2$).
If $c$ is periodic for $f$ of minimal period $n_1\ge 2$ we set
(by convention) $n_0=1$.
If $c$ is preperiodic or periodic for $f$,  we set
\begin{equation}
N_f:=n_0+n_1-1\ge 2\, .
\end{equation}
(If $c$ is periodic we have $N_f=n_1$.)
If $c$ is neither preperiodic nor periodic for $f$, we set $N_f=\infty$.

Define $J:=(-\infty , f(c)]$ and $\chi:\real \to \{0,1,1/2\}$ by
\begin{equation}\label{defchi}
\chi(x)=0 \mbox{ if } x \notin J\, , \quad \chi(x)=
1 \mbox{ if } x \in \mbox{int}\,  J\, , \quad
\chi(f(c))=\frac{1}{2} \, .
\end{equation}
The two inverse branches of  $f$,
a priori defined on $[f(a),f(c)]$ and $[f(b),f(c)]$,
may be extended to maps $\psi_{+}: J \to \real_-$
and $\psi_{-}:  \to \real_+$ in $C^r(J)$, with $\sup |\psi_{\sigma}'|<1$ for $\sigma=\pm$.
We set
\begin{equation}\label{llambda}
 \lambda_0= \lim_{n \to \infty} (\sup( |(f^{-n})'|))^{1/n}\, ,\quad
\Lambda_0=\lim_{n\to \infty}
(\sup |(f^n)'| )^{1/n} \, .
\end{equation}
and choose
$$
\lambda \in (\lambda_0, 1) \, , \quad
\Lambda > \Lambda_0 \, .
$$

\begin{definition}
Let $r \ge r_0 \ge 2$ be integers.
For  a piecewise expanding $C^{r}$ unimodal map
$f$, a  $C^{r_0,r}$ {\it perturbation} of $f$ is a family of piecewise expanding 
$C^{r}$ unimodal maps $f_t: I \to I$, $|t|<\epsilon$,
with $f_0=f$, and satisfying the following
properties: There exists a neighbourhood $\II_\sigma$ of
$I_\sigma$,   $\sigma=\pm$, so that 
the $C^{r}$ norm of the extension  of $f_t|_{I_\sigma}$
to $\II_\sigma$
is uniformly bounded  for small $|t|$, and so that
\begin{equation}\label{forweak}
\|(f-f_t)|_{\II_\sigma}\|_{C^{r-1}}=O(t)\, .
\end{equation} 
The map
$
(x,t) \mapsto f_t(x)\, , 
$ 
extends to a $C^{r_0}$ function on a neighbourhood
of $(I_+ \cup I_-)\times\{0\}$. 
The {\it infinitesimal deformation} of the perturbation $f_t$ is defined by
\begin{equation}
v= \partial_t f_t|_{t=0}\, .
\end{equation}
\end{definition}

Our assumptions imply that the infinitesimal deformation
satisfies $v(a)=v(b)=0$ and, if $f(c)=b$, also $v(c)=0$.

If  $f_t$ is a $C^{2,2}$
perturbation of   a piecewise expanding $C^{2}$ unimodal map, 
then each $f_t$ (for small
enough $t$)  admits an absolutely
continuous invariant probability measure
(see e.g. \cite{Ba} for references), with a density
$\rho_t$ which is of bounded variation. In fact,
there is only one absolutely continuous invariant
probability measure. Each $\rho_t$ is continuous
on the complement of the at most countable set
$\{ f^k_t(c)\mid k\ge 1\}$,
and it is supported in $[f^2_t(c), f_t(c)]\subset [a,b]$
(we extend it by zero on $\real$). 
If $f$ is good and mixing,
then $f_t$ is mixing and the absolutely continuous invariant measure  is mixing.
(If $f$ is mixing, but not good,  $f_t$ need not
be mixing.) 
In other words, assuming that $f$ is good and mixing
implies that $f$ is stably mixing (we do not claim the converse),
in addition, 
denoting by $| \varphi|_{L^1(Leb)}$  the $L^1(\real,\mbox{Lebesgue})$ norm of 
$\varphi$,
by
\cite[Prop. 7]{KL} (by  uniform Lasota-Yorke estimates, see \cite[Remarks 1, 5]{KL}), 
we  have
\begin{equation}\label{claim'}
 |\rho_t - \rho_0|_{L^1(Leb)} = 0(t \ln |t|) \, .
\end{equation}
If  $f$
is not good, the function $t \mapsto\rho_t$ need not be continuous.
(This is germane to the   fact that
mixing is not necessarily preserved if $f$ is not good. See \cite{Ke} for an illuminating multimodal example.)
See also Remark~\ref{needgood}.

\begin{remark}
Note that Ruelle's conjecture offers a candidate for the derivative
of 
\begin{equation}
\RR(t)=\int \varphi \, \rho_t \, dx
\end{equation} 
only if $\partial_t f_t|_{t=0}=X\circ f$.
(See also Remark~\ref{noXf}.)
\end{remark}

\begin{definition}
For  integers $r \ge r_0 \ge 2$, and a piecewise expanding $C^{r}$ unimodal map
$f$, a {\it $C^{r_0,r}$ perturbation of $f$ tangent to the topological class of $f$}
is a $C^{r_0, r}$ perturbation $f_t$ of $f$ so that
there exist a $C^{2,2}$ perturbation
$\tilde f_t$ of $f$  with 
$$\sup_x|\tilde f_t(x)-f_t(x)|=O(t^2)$$ 
and
homeomorphisms $h_t$  with $h_t(c)=c$ and
$\tilde f_t=h_t \circ f \circ h_t^{-1}$. 
\end{definition}

Clearly,
if $f_t$ is a $C^{2,2}$ perturbation of $f$ tangent to the topological class of $f$,
then $v=\partial_t f_t|_{t=0}=\partial_t \tilde f_t|_{t=0}$.
We shall see (Corollary~\ref{horiz}) that the infinitesimal deformations
$v$ of tangent perturbations are {\it horizontal} for $f$:

\begin{definition}
A continuous  $v : I \to \real$ is  {\it horizontal} 
\footnote{See \cite{Ly}, \cite{ALM} and references therein for a motivation of this
terminology.}
for a piecewise expanding $ C^1$ unimodal map $f$ if,
setting $M_f=n_1$ if $c$ is periodic of minimal period $n_1\ge 2$,
and $M_f=+\infty$ otherwise,
\begin{equation}\label{Mf}
 \sum_{j = 0}^{M_f-1} \frac{v(c_j)}{(f^{j})'(c_1)}=0 \, .
\end{equation}
\end{definition}

See also Subsection~\ref{families} for a discussion 
of perturbations  $f_t$ tangent to the topological class
of $f$. 

When considering  $C^{2,2}$ perturbations $f_t$, we have in particular
$\sup_x|f'_t(x)-f'(x)|=o(1)$ (considering the extensions
to neighbourhoods of $I_\sigma$) and
we shall implicitly restrict to $\epsilon$ small enough so that
\begin{align}\label{llambda2}
&\sup_{|t|< \epsilon}  \lim_{n \to \infty} (\sup( |(f_t^{-n})'|))^{1/n}
< \lambda \, ,\quad
\sup_{|t|< \epsilon} \lim_{n \to \infty} (\sup( |(\tilde f_t^{-n})'|))^{1/n}
< \lambda \, ,\\
\nonumber & \sup_{|t|< \epsilon} \lim_{n\to \infty}
(\sup |(f_t^n)'| )^{1/n}  < \Lambda\, ,
\quad \sup_{|t|< \epsilon} \lim_{n\to \infty}
(\sup |(\tilde f_t^n)'| )^{1/n}  < \Lambda\, .
\end{align}

\subsection{The twisted cohomological equation, the smooth motions
$h_t(x)$, and the infinitesimal conjugacy $\alpha$}

In this section, we discuss  the following
{\it twisted cohomological equation}
(TCE, see e.g. \cite{Sm})  for piecewise expanding  unimodal $f$ and bounded $v$:
\begin{equation}\label{tceeq}
v(x)=\alpha(f(x))- f'(x) \alpha(x)\, , \quad \forall x \in I \, , x\ne c\, .
\end{equation}

Let us start with an easy lemma:

\begin{lemma}\label{easyl}
Assume that $f$ is a piecewise expanding $C^1$ 
unimodal map  
and that $v$ is a bounded function on $I$.
Then for every $\omega\in \real$ the unique bounded
solution $\alpha_{(\omega)}$ to (\ref{tceeq}) which satisfies $\alpha_{(\omega)}(c)=\omega$
is given by:
\begin{equation}\label{aalpha}
\alpha_{(\omega)}(x)=
\begin{cases}
 - \sum_{j= 0}^\infty \frac{ v(f^{j}(x))}{(f^{j+1})'(x)}\, ,&
 \mbox{if } f^j(x)\ne c \, , \forall j \ge 0 \, ,\\
\frac{\omega}{(f^{\ell})'(x) } - \sum_{j= 0}^{\ell-1}\frac{ v(f^{j}(x))}{(f^{j+1})'(x)}&\mbox{if } \exists \ell \ge 1
\mbox{ s.t. } f^{\ell}(x)=c \, .
 \end{cases}
\end{equation}
\end{lemma}

\begin{remark}
If (\ref{tceeq}) admits a continuous solution $\alpha$,
it is easy to see by taking
limits as $x\to c$ from the
left and from the right that $\alpha(c)=0$ and $v(c)=\alpha(c_1)$.
(In particular, there is at most one continuous solution
to (\ref{tceeq}).)
We shall not use this.
\end{remark}

\begin{proof}
For $x$ so that $f^\ell (x)\ne c$ for all $\ell \ge 0$ 
(\ref{aalpha}) defines a bounded solution uniquely on this set:
Indeed any bounded solution satisfies $\beta=-v/f'- \ldots -
v\circ f^{k-1}/(f^k)'+ \beta \circ f^{k+1}/(f^k)'$;
if $\beta(x)\ne \alpha_{(\omega)}(x)$, then we take $k$ so that
$K/(f^k)' < (\beta(x)-\alpha_{(\omega)}(x))/3$ with $K =\max( \sup| \beta|, 
\sup| \alpha_{(\omega)}|)$,
and we get a contradiction.
If $\beta(c)=\omega$, then for each $x$ so that $f^{\ell}(x)=c$ we
must have $\beta(x)=\alpha_{(\omega)}(x)$ as defined in (\ref{aalpha}).
\end{proof}

When $v$ is the infinitesimal
deformation of a perturbation $f_t$ tangent
to the topological class of $f$ 
we shall  relate solutions to (\ref{tceeq}) to the conjugacies  $h_t$.
The key ingredient for this
is the following  information about the smoothness
of $t\mapsto h_t$:

\begin{proposition}\label{TCE}
Let $r_0\ge 2$ be an integer.
Assume that $\tilde f_t$ is a $C^{r_0,r_0}$ perturbation
of a  piecewise expanding $C^{r_0}$
unimodal  map $f$, so that for each small $t$ there
exists a homeomorphism $h_t$ with $h_t(c)=c$ and
$\tilde f_t = h_t \circ f \circ h_t^{-1}$. 
Then for  small enough $\epsilon$, the map
$(t,x)\mapsto h_t(x)$ is continuous from
$(-\epsilon,\epsilon)\times I\to \real$ and
the maps $t \mapsto h_t(x)$ 
are $C^{r_0-1+Lip}$ on $[-\epsilon, \epsilon]$, uniformly in $x\in  I$.
(I.e. $\sup_x \| h_{\cdot}(x)\|_{ C^{r_0-1+Lip}([-\epsilon, \epsilon])} < \infty$.)
\end{proposition}

\begin{remark}
Although the $h_t(x)$ cannot be called ``holomorphic motions" 
(see e.g. \cite{ALM}) they certainly
be called ``smooth motions"!
Beware that the maps $t\mapsto h_t^{-1}(x)$ are in general
not $C^{1+Lip}$, although it is easy to see
that the map $t\mapsto h_t^{-1}(x)$ is differentiable at $t=0$
with derivative $-\alpha(x)$ for all $x \in  I$.
Also, the maps $x\mapsto h_t(x)$, $x \mapsto h_t^{-1}(x)$ are in
general not absolutely continuous (see Section ~\ref{elegg}).
\end{remark}

It will then be easy to show:

\begin{corollary}\label{horiz}
Under the assumptions of Proposition~\ref{TCE}
the bounded function $\alpha: I\to \real$ defined by
$\alpha(x)=\partial_t h_t(x)|_{t=0}$ satisfies the TCE (\ref{tceeq})
for $v=\partial_t f_t|_{t=0}$.
In addition,   $\alpha$ is continuous, $\alpha(c)=0$ 
and $v(c)-\alpha(c_1)=0$, so that $v$ is  horizontal for $f$.
\end{corollary}

\begin{definition}
Under the assumptions of Proposition ~\ref{TCE},
the function $\alpha=\partial_t h_t|_{t=0}$ is the {\it infinitesimal conjugacy}
associated to the infinitesimal deformation $v$
of $f_t$.
\end{definition}

\begin{remark}
It follows from Corollary~\ref{horiz} that if 
$f_t$ is a perturbation of
$f$ and $v=\partial_tf_t|_{t=0}$ is not horizontal for $f$, then there exist
arbitrarily small $t$ so that
$f$ and $f_t$ are {\it not} topologically conjugated, in particular  $f$ is
not structurally stable.
See \cite{av} for an analogous statement about rational maps.
\end{remark}

\begin{proof}[Proof of Proposition~\ref{TCE}]
To simplify notation, we assume that $c=0$ in this proof.
Let $\mathcal{P}_t$ be the set of points which are either periodic or eventually periodic 
for  $\tilde f_t$, and whose forward orbit under $\tilde f_t$
does not contain the turning point $c$.  
It is easy to see that $\mathcal{P}_t$ is dense in $I$. 
Let $\theta=\sup_{x,t} |\tilde f'_t(x)|^{-1}$.
We first prove that 
$(t,x)\rightarrow h_t(x)$
is continuous. Fix $(x_0,t_0)$ and let $\kappa >0$. 
Pick  $n \in \mathbb{N}$ and $\delta > 0$ such that 
$\theta^n + \frac{\delta}{1-\theta}< \kappa
$.
Choose $\eta_0 < \epsilon/2$ small enough such that if $|t-t_0|< \eta_0$ then
$$\sup_x |\tilde f_{t}(x) -\tilde f_{t_0}(x)| < \delta\, ,$$
and let   $\eta_1$ be such that  $|x-x_0|< \eta_1$ implies  
$f^k(x)\cdot  f^k(x_0)\geq 0$,
for every $k\geq n$.  So
$\tilde f_t^k(h_t(x))\cdot  \tilde f_{t}^k(h_t(x_0))\geq 0$,
for every $k\geq n$ and $t$. Of course
$\tilde f_t^k(h_t(x_0))\cdot \tilde  f_{t_0}^k(h_{t_0}(x_0))\geq 0$.
By Lemma~\ref{conju}, for every 
$(t,x) \in \{ |t-t_0| <  \eta_0\} \times \ \{ |x-x_0|< \eta_1  \}$
we have 
$$
|h(t,x)- h(t_0,x_0)|\leq \kappa \, .
$$

In the remainder of this proof,  $\partial_t^i h_t$ 
denotes $\partial^i_s h_s|_{s=t}$. 
The implicit function theorem tells us that if $p \in \mathcal{P}_0$ then  
$t\rightarrow h_t(p)$ is a $C^{r_0}$  function. Differentiating the equation
$h_t\circ f(p)=\tilde  f_t\circ h_t(p)$
with respect to $t$ we obtain
\begin{equation}\label{tce1} \partial_t h_t\circ f(p) = \partial_t \tilde f_t\circ h_t (p) +
 \tilde f'_t(h_t(p))\partial h_t(p) \, .
 \end{equation}
In other words
\begin{equation*}\partial_t h_t\circ f(p) -\tilde f'_t(h_t(p))\partial h_t(p)= 
\partial_t \tilde f_t\circ h_t (p)= F_1(p) \, .
\end{equation*}
Next, differentiating   (\ref{tce1}) $r_0$  times,  
we can easily prove  that for each $i \leq r_0$ 
\begin{equation}\label{tce5}
\partial_t^i h_t\circ f(p) -\tilde f'_t(h_t(p))\partial_t^{i} h_t(p)= F_i(p)\, , 
\end{equation}
where the function $F_i$ is a polynomial combination of
compositions of (all)  partial derivatives  of $\tilde f_t(x)$ up to order $i$, 
including mixed ones, with the function $h_t$,  and 
partial derivatives $\partial_t^j h_t$, for $j=1,\dots, i-1$.

For every $q \in \mathcal{P}_t$, we have $q=h_t(p)$, with $p \in \mathcal{P}_0$. Define 
$$\alpha_t^i(q) := \partial_t^i h_t (h^{-1}_t(q)).$$
Define $Q_i(q)=F_i(h^{-1}_t(q))$. From (\ref{tce5}) we obtain the twisted cohomological equation
\begin{equation}\label{tce2} 
Q_i(q)=\alpha_t^i(\tilde f_t(q))- \tilde f'_t(q)\cdot \alpha_t^i(q)\, .
\end{equation}
Let call this equation $TCE_i$.

Note that $F_1$ is bounded on $\mathcal{P}_0$. We claim that 
$$|F_i|_{\infty} < \infty
$$
for every $i\leq r_0$. Indeed, suppose by induction  that 
$F_{\ell}$ and $\partial_t^{\ell-1} h_t$  are bounded functions on $\mathcal{P}_0$, for every $\ell\leq i < r_0$. Then  $Q_i$ is bounded on $\mathcal{P}_t$, and the 
unique solution for $TCE_i$  on $\mathcal{P}_t$  is given by the expression
$$
\alpha_t^i(q)= -\sum_{j=0}^{\infty} \frac{Q_i(\tilde f^j_t(q)) }{(\tilde f_t^{j+1})'(q)} \, .
$$
The uniqueness of the solution follows from the fact that every point in $\mathcal{P}_t$
is eventually periodic. 

In particular
\begin{equation}\label{equicont} 
\sup_{q \in \mathcal{P}_t}  |\alpha_t^i(q) |\leq 
\frac{|Q_i|_\infty}{1-\sup_{x} 
|\tilde f'_t(x)|^{-1}} \, .
\end{equation}
It follows that $\partial_t^i h_t$ is bounded on $\mathcal{P}_0$, and hence $F_i$ is bounded in the same domain. This concludes the inductive argument.

Then from (\ref{equicont}) we have an  upper bound  for $|\partial_t^i  h_t|$,
for  $i \leq r_0$,  which is uniform on $t \in [-\epsilon, \epsilon]$
(up to taking a smaller $\epsilon$). So the family  of functions 
$t\rightarrow h_t(p)$,
with $p \in \mathcal{P}_0$ and $t \in [-\epsilon, \epsilon]$, is  a  bounded subset  
of $C^{r_0}([-\epsilon, \epsilon])$.

We claim that $t\mapsto h_t(x)$ is $C^{r_0-1+Lip}$ for every  $x \in I$. Indeed,
let $p_n \in \mathcal{P}_0$ be a sequence which converges to $x$.  
Of course the  sequence of functions $t\mapsto h_{t}(p_n)$ converges to the
function $t\mapsto h_t(x)$. Since every sequence in a bounded subset of 
$C^{r_0}([-\epsilon,\epsilon])$ has a subsequence 
which converges to a function in $C^{r_0-1+Lip}$, 
we conclude that $t\mapsto h_t(x)$ is $C^{r_0-1+Lip}$.
\end{proof}

\begin{proof}[Proof of Corollary \ref{horiz}]
By differentiating
$\tilde f_t \circ h_t = h_t \circ f$ with respect  to $t$ at $t=0$, we see that
$\alpha(x)$ satisfies (\ref{tceeq}) at all $x \ne c$.
Since $h_t(c)=c$ for all $c$ we have  $\alpha(c)=0$.
To prove $v(c)=\alpha(c_1)$, we use  $\tilde f_t \circ h_t(c) = h_t \circ f(c)$:
The derivative with respect to $t$ of the right-hand-side at $t=0$ is just $\alpha(c_1)$.
This implies that the left-hand-side is differentiable at $t=0$, and,
using $h_t(c)=c$, the derivative is
$$
\lim_{t\to 0}\frac {\tilde f_t(h_t(c))-\tilde f_t(c)}{t}+
\lim_{t\to 0}\frac {\tilde f_t(c)-f(c)}{t}=0+ v(c)\, .
$$
\end{proof}


\subsection{Perturbations $f_t$ tangent to the topological class of $f$}
\label{families}

For $r \ge 2$ and a fixed  piecewise expanding $C^r$
unimodal map $f$, we may  pick  $h_t(x)$ with
$h_t(c)=c$, so
that  $(x,t)\mapsto h_t(x)$ is $C^{r}$,
and define $\tilde f_t := h_t \circ f \circ h_t^{-1}$.
Then $\tilde f_t$ is a $C^{r,r}$ perturbation of $f$ in its topological class.
If we assume in addition that
 $h_t(c+x)=\SS h_t(c-x)$, where the ($C^r$) symmetry
$\SS$ is such that $f(c+x)=f(\SS(c-x))$,
we can ensure that the infinitesimal deformation is of the form
$v=X\circ f$.
Since $x\mapsto h_t(x)$ is a diffeomorphism in this construction,  it gives a conjugacy
between the invariant densities $\tilde \rho_t$ 
of $\tilde f_t$ and $\rho_0$ of $f$. Thus differentiability of 
$
\widetilde \RR(t)
=\int \varphi \tilde \rho_t\, dx
$
can be obtained by relatively easy perturbation
theory arguments on the transfer operator.
Theorem~\ref{formula} applies to {\it all} smooth perturbations $f_t$
which are tangent to  $\tilde f_t$, and we may choose $f_t$
in such a way as to ensure that $f_t$ and $f$ are not
topologically conjugated (by modifying the kneading
invariant), or are not smoothly conjugated
(by acting on the multipliers \cite{MaM}). 

In view of a more general and
systematic description of perturbations tangent
to the topological class, recall that
Corollary ~\ref{horiz} implies that
if a $C^{2,2}$ perturbation $f_t$ of a  $C^{2}$ map $f$ is tangent to the topological 
class of $f$, then its infinitesimal deformation
$v$ is horizontal.
In the
smooth nonuniformly hyperbolic case (see \cite{Ly}, \cite{ALM} and references therein) a converse to this
statement holds.
The proof of the  converse in our setting is given elsewhere:

\begin{theorem}\label{???}(See \cite{??})
For $r_0\ge 2$, let $f$ be a good piecewise expanding $C^{r_0}$
unimodal map and let
$v\in C^{r_0}(I)$  be horizontal for $f$ and satisfy
$v(a)=0$, $v(b)=0$, and, if $f(c)=b$, also $v(c)=0$. Then
there exists a family of piecewise expanding 
$C^{r_0}$ unimodal maps $\tilde f_t: I \to I$, $|t|<\epsilon$,
with $\tilde f_0=f$, so that the map
$
(x,t) \mapsto \tilde f_t(x)\, , 
$ 
extends to a $C^{r_0-1+Lip}$ function on a neighbourhood
of $(I_+ \cup I_-)\times\{0\}$, and, in addition,
$\partial_t \tilde f_t|_{t=0}=v$, and for each $t$
there is a homeomorphism $h_t$ with
$h_t(c)=c$ and  $\tilde f_t=h_t\circ  f \circ h_t^{-1}$.
The conjugacies  $h_t$  are in general not
absolutely continuous.
\end{theorem}

In particular, the above implies that
any  $C^{2,r}$ perturbation $f_t$ of a  piecewise
expanding $C^r$ unimodal map $f$ ($r\ge 2$)  so that
$v=\partial_t f_t|_{t=0}$ is horizontal and $v\in C^{2}(I)$ 
is tangent to  the topological class of $f$. 

Note that there exist (many)  $C^{2,r}$ perturbations $f_t$ 
of  mixing piecewise expanding $C^r$ unimodal maps,
and such that $v=\partial_t f_t|_{t=0}$ is $C^r$ and horizontal (also if we require
$v=X\circ f$). 
Indeed, the functional $L_f:v \mapsto
v(c)-\alpha_{(0)}(c_1)$ is bounded and linear from $\{ v \in C^r(I) \}$ to $\real$. So 
it has a codimension-one kernel.

\section{Transfer operators and their spectra}
\label{three}

\subsection{Definitions and previous results}

Recall that a point $x$ 
is called regular  for a function $\phi$ 
if $2\phi(x)=\lim_{y \uparrow x} \phi(y)+\lim_{y\downarrow x}\phi(y)$.
If $\phi_1$ 
and $\phi_2$ are  functions of
bounded variation on $\real$
having at most regular discontinuities, the Leibniz formula
says that $(\phi_1 \phi_2)'=\phi_1'\phi_2+\phi_1 \phi_2'$, where both
sides are a priori finite measures.
(Viewing a function $\phi$ in $BV$ as a measure means considering
$\phi\, dx$.)

\smallskip

For a  piecewise expanding $C^{2}$ unimodal map $f$,
recalling (\ref{defchi}), we  introduce two linear operators:
\begin{equation}\label{L0}
\LL_{0} \varphi(x):=
\chi(x) \varphi(\psi_{+}(x))-\chi(x)  \varphi(\psi_{-}(x))\, ,
\end{equation}
and
\begin{equation}\label{L1}
\LL_{1} \varphi(x):=
\chi(x)
\psi'_{+}(x)\varphi(\psi_{+}(x))
+\chi(x)
|\psi'_{-}(x)| \varphi(\psi_{-}(x))\, .
\end{equation}
Note that $\LL_{1}$
is  the usual (Perron-Frobenius) transfer operator for $f$,
in particular, $\LL_{1} \rho_0=\rho_0$ and $\LL_{1}^*(\mbox{Lebesgue}_\real)
=\mbox{Lebesgue}_\real$.
The operators $\LL_{0}$ and $\LL_{1}$  both act boundedly on 
the Banach space
$$
BV=BV^{(0)}
:=\{ \varphi : \real \to \complex \mid
\var (\varphi) < \infty\, , \supp (\varphi) \subset [a , b]\} / \sim\, ,
$$ 
endowed with the norm $\| \varphi\|_{BV} =\inf _{\phi \sim \varphi} \var (\phi)$,
where $\var$ denotes  total
variation and $\varphi_1 \sim \varphi_2$ if the
bounded functions $\varphi_1$, $\varphi_2$ 
differ on an at most countable set. 
To get finer information on $\LL_0$, we consider the smaller Banach space
(see e.g.  \cite{Ru1})
$$
BV^{(1)}=\{ \varphi : \real \to \complex\mid
\supp(\varphi)\subset (-\infty, b]\, ,
\varphi' \in BV \}\, ,
$$
for the norm
$\|\varphi\|_{BV^{(1)}} =  \|\varphi'\|_{BV}$.
If $\LL$ is a bounded linear operator on a Banach space $\BB$, we denote
the spectrum of $\LL$ by $\sp(\LL)$, and
we define $R_{ess}(\LL)$, the essential spectral radius of $\LL$, to be
\begin{align*}
R_{ess}(\LL)&=\inf \{ R \ge 0 \mid  \sp(\LL)\cap \{ |z| > R\} \\
&\qquad\qquad\qquad\qquad \mbox{ consists of isolated
eigenvalues of finite multiplicity}\} \, .
\end{align*}
Recalling the definition (\ref{llambda}) of
$\lambda_0$, we have the following key lemma (see \cite{Ba}, the claims
on $\LL_1$ on $BV$ are  classical):

\begin{lemma}\label{sp}
Assume that $f$ is a mixing piecewise expanding
$C^2$ unimodal map. The
essential spectral radius of $\LL_{1}$ on $BV$ is $\le \lambda_0$. In addition,
$1$ is a maximal eigenvalue of $\LL_{1}$, which is simple, for the eigenvector
$\rho_0$, and there are no other eigenvalues of $\LL_1$ of modulus $1$ on $BV$.
The spectral radius of $\LL_0$ on $BV$ is equal to $1$.
For any $\varphi \in BV^{(1)}$ 
\begin{equation}\label{keye}
(\LL_{0} \varphi)'=\LL_{1} (\varphi')\, .
\end{equation}
Finally, the spectrum of $\LL_{0}$ on $BV^{(1)}$  and that of
$\LL_{1}$ on $BV$ coincide. 
\end{lemma}

For further use, associate to a mixing piecewise expanding
$C^2$ unimodal map $f$
\begin{equation}\label{ttau}
\tau_0= \max \biggl ( \lambda_0\, , \, \, 
\sup \{ |z| \mid z \in \mbox{sp} \, (\LL_1|_{BV})\, , \, \, z\ne 1\} \biggr ) \, ,
\end{equation}
(note that $\tau_0 <1$), and choose
$$
\tau \in (\tau_0, 1) \, .
$$

Set $H_u(x)=-1$ if $x < u$,
$H_u(x)=0$ if $x > u$ and $H_u(u)=-1/2$.
If $f$ is a piecewise expanding $C^2$ unimodal map,
the invariant density of $f$ is of bounded variation
and thus decomposes uniquely \cite{RN} as $\rho_0=\rho_{sal}+\rho_{reg}$ 
with $\rho_{reg}$ continuous
and $\rho_{sal}$ the {\it saltus term}
(recalling $N_f$ from \S~\ref{2.1}):
\begin{equation}\label{dec0}
\rho_{sal} = \sum_{n = 1}^{N_f} s_n H_{c_n}\, ,
\end{equation}
with 
$s_{n}=\lim_{y \downarrow c_{n}} \rho_0(y)-\lim_{x \uparrow c_{n}} \rho_0 (x)$.
By
\cite[Prop. 3.3]{Ba} we have\footnote{The proof there does not require
that $c$ is not periodic.}:

\begin{proposition}\label{decc}
Let  $f$ be a  mixing piecewise expanding $C^3$ unimodal map.
Then $\rho_{reg}$ from the decomposition
(\ref{dec0}) of the invariant density 
is an element of $BV^{(1)}$.
\end{proposition}

(M. Misiurewicz pointed out to us the related work of \cite{Sz}.)

\subsection{Comparing the invariant densities of two tangent perturbations}

Our main result is about perturbations $f_t$ which are
tangent to the topological class of $f_0$.  In this subsection,
we prove Proposition~\ref{fromKL} (using classical
Banach spaces, and tools from Keller-Liverani \cite{KL})
which will allow us to reduce from this
assumption to the hypothesis that $f_t$ lies in the topological class
of $f_0$.

We need more notation.
Let $f_t$ be a $C^{2,r}$ perturbation of a
piecewise expanding $C^r$ unimodal map ($r \ge 2$)
Define $J_t:=(-\infty , f_t(c)]$ and $\chi_t:\real \to \{0,1,1/2\}$ by
$$
\chi_t(x)=
0 \mbox{ if } x \notin J_t\, , \quad
\chi_t(x)=1 \mbox{ if }   x \in \mbox{int}\,  J_t\, , \quad
\chi_t(f_t(c))=\frac{1}{2}\, .
$$
The two inverse branches of  $f_t$,
a priori defined on $[f_t(a),f_t(c)]$ and $[f_t(b),f_t(c)]$,
may be extended to  maps $\psi_{t,+}: J_t \to (-\infty, c]$
and $\psi_{t,-}: J_t \to [c, \infty)$ in $C^r(J_t)$, with $\sup |\psi_{t,\sigma}'|<1$ 
for $\sigma=\pm$.
Put
\begin{equation}\label{L1t}
\LL_{1,t} \varphi(x):=
\chi_t(x)
\psi'_{t,+}(x)\varphi(\psi_{t,+}(x))
+\chi_t(x)
|\psi'_{t,-}(x)| \varphi(\psi_{t,-}(x))\, .
\end{equation}

Recall our choices $\lambda<1$ from (\ref{llambda})
and $\tau <1$ from (\ref{ttau}). Lemma~\ref{sp} applies
to $\LL_{1,t}$. By \cite{KL} we may assume that $t$ is small enough so that
$$
\max \bigl ( \lambda, \, 
\sup_t \sup \{ |z| \mid z \in \mbox{sp} \, (\LL_{1,t}|_{BV})\, , \, \, z\ne 1\}\} 
\bigr ) < \tau\, .
$$

We may now state the new result of this subsection:

\begin{proposition}\label{fromKL}
Let $f$ be a good mixing piecewise expanding $C^2$ unimodal map.
Then  for any $C\ge 1$ and every pair
$(f_t, g_t)$ of $C^{2,2}$ perturbations of $f$, and so that
\begin{equation}\label{tangt}
\sup_x | f_t(x)-g_t(x)|\le C t^2 \, ,\, \,  \forall |t|\le \epsilon\, ,
\end{equation}
there exist $C_1 \ge 1$,
 $\epsilon_0 >0$ and $\xi >1$ so that,
letting $\rho_t$  and $\tilde \rho_t$ denote the respective invariant densities of $f_t$
and $g_t$, we have
$$
\| \rho_t - \tilde \rho_t \|_{L^1(Leb)} \le C_1|t| ^\xi\, , \quad \forall |t|\le \epsilon_0 \, .
$$
\end{proposition}

\begin{remark}\label{needgood}
The assumption that $f$ is good 
is crucial in the above proposition since otherwise
we do not have uniform Lasota-Yorke bounds (\ref{B}) in general.
\end{remark}

\begin{proof}
Recall  $\lambda <1$ 
from (\ref{llambda}) (we require that (\ref{llambda2})
hold for $g_t$ too).
Denote by $\LL_{1,t}$ the transfer
operator of $f_t$, by $\widetilde \LL_{1,t}$ the transfer operator of
$g_t$,  acting on $BV$.
Each $\LL_{1,t}$ and each $\widetilde \LL_{1,t}$ has
a simple maximal eigenvalue at $z=1$ and essential spectral radius $\le \lambda$
for small enough $t$. 
Our assumptions ensure that
\begin{equation}\label{unif}
\| f_t(x)\|_{C^{1+Lip}(V)}\le C\, ,
\quad \| g_t(x)\|_{C^{1+Lip}(V)}\le C\, ,
\end{equation}
on a neighbourhood $V$ of $(I_+\cup I_-) \times \{0\}$.
Also, there exist $\widetilde C$ and $\epsilon_1>0$
depending only on $f$ and $C$ so that
(our assumptions imply that $g_t$ and $f_t$
satisfy (\ref{forweak}))
\begin{align}\label{A}
&\sup_j \|\LL^j_{1,t}\|_{L^1(Leb)} < \widetilde C
\,  , \quad
 \sup_j \|\tilde \LL^j_{1,t} \|_{L^1(Leb)}< \widetilde C \, , \, \, \forall |t|\le \epsilon_1 \, ,
\\
\nonumber &
\| \LL_{1, t}( \varphi) - \LL_{1,0}( \varphi) \|_{L^1(Leb)} \le \widetilde C |t|\|\varphi\|_{BV}\, ,
 \, \, \forall \varphi \in BV\, , \, \, 
  \forall |t|\le \epsilon_1\, ,
\\
\nonumber &
\|\tilde \LL_{1, t}( \varphi) - \LL_{1,0}( \varphi) \|_{L^1(Leb)} \le \widetilde C |t|\|\varphi\|_{BV}\, ,
 \, \, \forall \varphi \in BV\, , \, \, 
  \forall |t|\le \epsilon_1 \, .
\end{align}
also, since  $f$ is good \cite[Remark 5]{KL},
\begin{equation}\label{B}
\max ( \|\LL^j_{1,t} \varphi\|_{BV} ,
\|\tilde \LL^j_{1,t}\|_{BV}) 
\le \widetilde C \lambda^j \|\varphi\|_{BV} + \widetilde C \|\varphi\|_{L^1}\, ,
 \, \, \forall \varphi \in BV\, , \, \,  \forall |t|\le \epsilon_1 \, ,
\end{equation}
finally, (\ref{unif}) and (\ref{tangt}) imply $\|(f_t-g_t)|_{I_\sigma}\|_{C^1}=O(t^2)$,
with a constant depending only on $f$ and $C$, and thus
\begin{equation}\label{tgt2}
 \|\LL_{1,t} (\varphi )- \tilde \LL_{1,t} (\varphi) \|_{L^1(Leb)} \le \widetilde C t^2\|\varphi\|_{BV}\, ,
 \, \,  \forall \varphi \in BV\, , \, \,  \forall |t|\le \epsilon_1
\, . 
\end{equation}

It follows from (\ref{A}--\ref{B}) for
$\widetilde \LL_{1,t}$, $\LL_1$, and \cite[Theorem 1]{KL} that 
for each small enough $\delta >0$ there are $\epsilon_2>0$
and $\widehat C \ge 1$, depending
only on $f$ and $C$ so that
\begin{equation}\label{tildeok}
  \| (z-\widetilde \LL_{1,t})^{-1} \|_{BV} \le \widehat C \, , \, 
  \forall |t|\le \epsilon_2\, ,
  \forall z \mbox{ with } |z|\ge \tau+\delta\mbox{ and } |z-1|\ge \delta\, .
  \end{equation}
We claim that the above estimate together
with (\ref{tgt2}) implies $\|\rho_t - \tilde \rho_t\|_{L^1(Leb)}=O(|t|^{2\eta})$ for
any $\eta < 1$.
Taking $\eta$ so that $2\eta >1$, the claim ends the proof.

To obtain the claim, we revisit the proof of \cite[Theorem~1]{KL}.
Following Keller--Liverani, we put $\QQ_t=(z-\LL_{1,t})$ and $\widetilde \QQ_t=(z-\widetilde \LL_{1,t})$.
In the sense of formal power series in $z$, we have for all $|t|\le \epsilon$
\begin{equation}\label{resolv}
\QQ_t^{-1} - \widetilde \QQ_t^{-1}=
 \QQ_t^{-1}(\LL_{1,t}-\widetilde \LL_{1,t}) \widetilde \QQ_t^{-1}\, .
\end{equation}
By  (\ref{tildeok}) and (\ref{tgt2}), the second part of the
proof of \cite[Theorem~1]{KL} 
gives that for any $\eta <1$ and $\gamma >0$, there are constants
$\epsilon_0>0$,
$\widetilde A\ge 1$, $\widetilde B\ge 1$, depending only on
$\eta$, $\widetilde C$  and $\gamma$, so that
for  any $z$ satisfying  $|z|\ge \tau+\gamma$ and $|z-1|\ge \gamma$,
all $\varphi \in BV$, and all $|t|\le \epsilon_0$,
\begin{align}
\|\QQ_t^{-1}(\varphi)\|_{L^1(Leb)}
&\le 2 (t ^2)^\eta
\bigl (\widetilde A \| \widetilde \QQ_t^{-1}\|_{BV}+ \widetilde B)\|\varphi\|_{BV}\\
\nonumber &\qquad\qquad
+2
(t^2)^{\eta-1} \biggr (\widetilde C \|\widetilde \QQ_t^{-1}\|_{BV}+ \frac {\widetilde C}{1-\tau}\biggl )
\|\varphi\|_{L^1(Leb)} \, .
\end{align}
Applying the above estimate to $(\LL_{1,t}-\widetilde \LL_{1,t}) \widetilde \QQ_t^{-1}(\varphi)$ and using (\ref{resolv}), we get
\begin{align}
\nonumber \| (\QQ_t^{-1}&- \widetilde \QQ_t^{-1}) (\varphi)\|_{L^1}\\
&\le  2 |t| ^{2\eta}( \| \LL_{1,t}\|_{BV}+\|\widetilde \LL_{1,t}\|_{BV})
\bigl (\widetilde A \| \widetilde \QQ_t^{-1}\|_{BV}+ \widetilde B) \| \widetilde \QQ_t^{-1}\|_{BV}\|\varphi\|_{BV}\\
\nonumber &\qquad
+2
C |t|^{2\eta} \biggr (\widetilde C \|\widetilde \QQ_t^{-1}\|_{BV}+ \frac {\widetilde C}{1-\tau}\biggl )
 \| \widetilde \QQ_t^{-1}\|_{BV}\|\varphi\|_{BV} \, ,
\end{align}
for any $\varphi \in BV$.
Writing the difference between the spectral projectors for the eigenvalue $1$
of $\LL_{1,t}$ and $\widetilde \LL_{1,t}$ as a contour integral of the 
difference of the resolvents, this shows the
claim.
\end{proof}

\subsection{Spaces of sums of smooth functions and postcritical jumps}
\label{3.2}

In this subsection we shall introduce Banach spaces $\BB_{t}\subset BV$
and $\BB^{Lip}_t\subset BV$ of functions with controlled jumps 
along the postcritical orbit, on which the transfer operators
$\LL_{1,t}$ have essential spectral radius $\le \lambda$, in view of the proof of
our main theorem in Section~\ref{four}.

Let $f$ be a mixing piecewise expanding $C^3$ unimodal map.
Recall that $N_f=n_0+n_1-1$ if $c$ is preperiodic, $N_f=n_1$ if
$c$ is periodic, and $N_f=\infty$ otherwise.
Let $\widetilde {BV}$ be the Banach space of continuous functions of
bounded variation supported in $[a,b]$, for the $BV$ norm. 
Fix $\eta >0$ small. 
Consider the Banach space $(\widehat \BB, \|\cdot\|)$ of pairs $\phi=(\phi_{reg}, \phi_{sal})$
with $\phi_{reg}\in \widetilde {BV}$,
and $\phi_{sal}=(u_k)_{k =1, \ldots, N_f }$,
normed by
\begin{equation}\label{clear}
\|\phi\|= \|\phi_{reg}\|_{BV} + | \phi_{sal} |_\eta
\mbox{ with }|\phi_{sal}|_\eta =\sup_{1\le k\le N_f}  (1+\eta)^k  |u_k|\, ,
\end{equation}
and so that, in addition,
\begin{equation}\label{bdok}
\phi_{reg}(x)=\sum_{k= 1}^{N_f} u_k \, , \, 
\forall x < a\, .
\end{equation}
We define 
$\Gamma=\Gamma_0 : \widehat \BB \to BV$ by 
\begin{equation}
\Gamma(\phi_{reg}, (u_k)_{k\ge 1})=\phi_{reg} + \sum_{k = 1}^{N_f} u_k H_{c_k} \, .
\end{equation}
(In particular, $\supp (\Gamma(\phi)) \subset [a,b]$.) 
The map $\Gamma$ is injective, and we define $\BB_{0} \subset BV$ to be
the isometric image of $\widehat \BB$ under $\Gamma$.

It is easy to see that $\rho_0 \in \BB_{0}$.
For $\phi=(\phi_{reg} , (u_k)_{k\ge 1}) \in   \widehat \BB$, 
we may decompose $\tilde \varphi=\LL_1(\Gamma(\phi))\in BV$
into $\tilde \varphi=\tilde \varphi_{reg}+\tilde \varphi_{sal}$.
Then, we have 
$$
\tilde \varphi_{sal}=\sum_{k\ge 1} w_k  H_{c_k} \, ,
$$ 
with (writing $f'(c_-)=\lim_{y\uparrow c} f'(y)$ and  $f'(c_+)=\lim_{y\downarrow c} f'(y)$)
\begin{equation}\label{spart'}
\begin{cases}
w_k= \frac{u_{k-1}}{f'(c_{k-1})}\, , & k \ge 2\, , \\
w_1 = -(\frac{1}{ |f'(c_-)|} + \frac{1}{ |f' (c_+)|})
\bigl (\phi_{reg}(c)+\sum_{k\ge 1, c_k >c} u_k\bigr ) \, ,\\ 
\end{cases}
\end{equation}
if the postcritical orbit is infinite (i.e., $N_f=\infty$), while
\begin{equation}\label{spart}
\begin{cases}
w_k= \frac{u_{k-1}}{f'(c_{k-1})}\, , & 2 \le k \le N_f\, , k \ne n_0 \\
w_{n_0}= \frac{u_{n_0-1}}{f'(c_{n_0-1})}+
 \frac{u_{n_0+n_1-1}}{f'(c_{n_0+n_1-1})}\, ,& \mbox{if } n_0 \ne 1\, ,  \\
w_1 = -(\frac{1}{ |f'(c_-)|} +\frac{1}{ |f' (c_+)|})
\bigl( \phi_{reg}(c)+\sum_{k\ge 1, c_k >c} u_k\bigr ) \, ,\\ 
\end{cases}
\end{equation}
if $N_f< \infty$.  Also, we find 
\begin{align}\label{rpart}
\tilde \varphi_{reg}& = \LL_1 (\phi_{reg}) \\
\nonumber &\qquad
+H_{c_1} 
\biggl( \frac {1}{ |f' (c_-)|} + \frac {1}{  |f' (c_+)|}\biggr )
\cdot \biggl (\phi_{reg}(c) +\sum_{1\le k\le N_f, c_k > c}   u_k\biggr )
\\
\nonumber &\qquad
+\sum_{k= 2}^{N_f} u_{k-1}
\biggl (\LL_1(H_{c_{k-1}}) - \frac{H_{c_k}}{f'(c_{k-1})}\biggr )\, .
\end{align}
It is thus not difficult
to check that  $\tilde \varphi\in \BB_{0}$.
We next prove that in fact $\LL_1$ is bounded on
$\BB_{0}$ with essential spectral radius $\le \lambda$.

We shall use   that if $\LL$ is a bounded
operator on a Banach space $\BB$, and $\KK$ is a compact
operator on $\BB$, then the essential spectral radii 
of $\LL$ and $\LL-\KK$ coincide
(see e.g. \cite{DS} or \cite[Theorem IV.5.35]{Ka}).
This fact is behind  most techniques to estimate the essential
spectral radius: Lasota-Yorke or Doeblin-Fortet bounds, Hennion's theorem,
the Nussbaum formula, see e.g. \cite{Bb}.
In view of this, recall that  the $BV$-closed unit ball is 
compact for the $L^1(Leb)$ norm.
(See e.g. \cite[\S3.2,
Prop. 3.3]{Bb}  for a proof of this Arzel\`a-Ascoli type result).
In view of obtaining compact perturbations if $N_f=\infty$, note that for any $\delta>0$
there is $k_\delta =O(\ln(\delta^{-1}))$ 
so that for any $\phi =(\phi_{reg}, (u_k)_{k\ge 1})\in \widehat \BB$, 
\begin{equation}\label{cpct}
\sum_{k\ge k_\delta}|u_k| \le \delta \sup_{k\ge 1}\bigl ((1+\eta)^k |u_k| \bigr ) \, .
\end{equation}

For $\varphi \in BV$,
we write $\Pi_{reg}(\varphi)=\varphi_{reg} \in C^0$ and $\Pi_{sal}(\varphi)=\varphi_{sal}$.
If $N_f\ne \infty$, the operator
$\KK_0(\varphi )=\Pi_{sal}(\LL_1(\varphi))
$ 
is finite rank on $\BB_0$, and thus compact. 
If $N_f=\infty$, the operator
$$\KK_0(\varphi )=-H_{c_1}( \varphi_{reg}(c)+\sum_{k\ge 1, c_k >c} u_k)
( |f'(c_-)|^{-1} +  |f' (c_+)|^{-1})
$$ 
is rank one, and thus compact, while
the operator $\Pi_{sal}\circ  (\LL_1-\KK_0) $
has norm bounded by $(1+\eta) \sup |f' |^{-1}$ by definition. 

We   next consider $\Pi_{reg} \circ \LL_1$.
If $N_f< \infty$, the second and third lines of (\ref{rpart}) are finite rank 
contributions, which will be denoted by
 $\KK_1(\phi)$.  If $N_f=\infty$, since
$$
\sup_{k\ge 2}
\biggl  \| \LL_1(H_{c_{k-1}}) - \frac{H_{c_k}}{f'(c_{k-1})}
\biggr \|_{BV} < \infty\, ,
$$
then (\ref{cpct})
implies that the second and third line of (\ref{rpart}) give a compact
contribution, also denoted by $\KK_1(\phi)$.

Then, consider the Radon measure $(\Pi_{reg} \circ \LL_1 (\varphi)-\KK_1(\varphi))'$.
By the Leibniz formula
we have, as Radon measures,
\begin{align}\label{Leibagain}
\nonumber (\Pi_{reg} \circ \LL_1( \varphi)-\KK_1(\varphi))'(y)&=\chi_J
\biggl ( \frac{ f''(\psi_+(y))}{(f'(\psi_+(y)))^2} \varphi(\psi_+(y))
-\frac{ f''(\psi_-(y))}{(f'(\psi_-(y)))^2} \varphi(\psi_-(y))\\
 &\qquad \qquad+ \frac{\varphi'(\psi_+(y))} {(f'(\psi_+(y)))^2} 
- \frac{\varphi'(\psi_-(y))} {(f'(\psi_-(y)))^2} \biggr )\, .
\end{align}
By the compact inclusion property
mentioned above, the contribution $\varphi_1$ in the first line is compact, 
let us call $(\KK_2(\varphi))'=\varphi_1$ the
corresponding operator. 
Now, the operator $\varphi' \mapsto \MM(\varphi')=
(\Pi_{reg} \circ \LL_1 (\varphi)-\KK_1(\varphi)-\KK_2(\varphi))'$ is bounded on measures,
with norm at most $\sup (|f'|^{-1}) \|\LL_1\|_\infty$
where $\|\LL_1\|_\infty$ is the operator norm of
$\LL_1$ acting on bounded functions. Applying the above argument to $\LL_1^j$, 
and using $\sup_j \|\LL^j_1\|_\infty < \infty$, we obtain  for each $j\ge 1$ a decomposition
$\LL_1^j =\KK^{(j)} + \MM^{(j)}$
where $\KK^{(j)}$  is  compact on $\BB_{0}$,  and
$ \|\MM^{(j)}\|_{\BB_0}\le  C_0 (1+\eta)^j \sup (|(f^j)'|^{-1})$.
Therefore, the essential spectral radius of $\LL_1$ on
$\BB_{0}$ is  $\le \lambda$. 

Consider now the Banach space 
$(\widehat \BB^{Lip}, \|\cdot\|)$ of pairs $\phi=(\phi_{reg}, \phi_{sal})$
with $\phi_{reg}\in Lip((-\infty,b])$,
and $\phi_{sal}=(u_k)_{k =1, \ldots, N_f }$,
normed by
$
\|\phi\|= \|\phi_{reg}\|_{Lip} + | \phi_{sal} |_\eta$
and so that
$\phi_{reg}(x)=\sum_{k= 1}^{N_f} u_k$
for all $x < a$ (in particular, $\phi_{reg}$ is constant
on $(-\infty,a)$). Using $\Gamma$ as above,
we define a Banach space $\BB_0^{Lip}\subset \BB_0\subset BV$.
Since $\|\phi\|_{Lip}= \|\phi'\|_{L^\infty}$ 
and since the $Lip([a,b])$- closed unit ball is compact in the
$L^\infty([a,b])$ topology, the same argument as above shows that
$\LL_1$ is bounded on $\BB_0^{Lip}$, with essential spectral
radius $\le \lambda$.
Since $BV^{(1)}\subset Lip$, we have that $\rho_0 \in \BB_0^{Lip}$.

If $f_t$ is a $C^{2,3}$ perturbation of $f$ we may
define $\BB_{t}$ and $\BB^{Lip}_t$ for each $t$ by taking the isometric
image in $BV$ of  $\widehat \BB$, respectively $\widehat \BB^{Lip}$ under 
$\Gamma_t$ defined by
$$
\Gamma_t \biggl (\phi_{reg}, (u_k)_{k \ge 1})\biggr ) = \phi_{reg}+
\sum_{k =1} ^{\infty} u_k H_{c_{k,t}} \, .
$$
The argument above shows that
$\LL_{1,t}$ has essential spectral radius bounded by $\lambda$ on
$\BB_{t}$ and $\BB_t^{Lip}$.
Since each $\BB_{t}$ and each $\BB^{Lip}_t$ is a subset of $BV$ and since $\rho_t\in \BB_t^{Lip}\subset \BB_t$, we have
proved that outside of the disc of radius $\tau$ the spectrum of
$\LL_{1,t}$ on $\BB_t$ or on $\BB^{Lip}_t$ consists in a simple eigenvalue at $1$, with
corresponding spectral projector $\varphi\mapsto \rho_t \int \varphi\, dx$.


\section{The susceptibility function and the candidate $\Psi_1$ for the derivative}
\label{fourr}

The susceptibility function \cite{Ru3} associated to  a piecewise expanding
$C^2$ unimodal map $f$, a test function $\varphi\in C^1([a,b])$,
and a deformation $v=X\circ f$ for  $X \in C^1([a,b])$ is   the formal power series
\begin{align}\label{susceptibility}
\Psi(z)&=\sum_{n=0}^\infty
\int z^n  X(y) \rho_0(y)( \varphi \circ f^n )' (y)\, dy =\sum_{n=0}^\infty
\int z^n \LL^n_0( X \rho_0)(x) \varphi' (x)\, dx\, .
\end{align}
In this section, we recall in Proposition~\ref{candidate} 
the resummation $\Psi_1$ proposed
in \cite{Ba} for the a priori divergent series $\Psi(1)$ when $X \circ f$
is horizontal.  In addition, we give in Lemma~\ref{reallynice} an expression for $\Psi_1$
in terms of the infinitesimal conjugacy $\alpha$ from Section~\ref{two}, and
we show that $\Psi_1$ is not well-defined if $X\circ f$ is not horizontal
(Proposition~\ref{nono}).

\begin{remark}\label{noXf}
If the infinitesimal deformation $v$ is not of the form $X \circ f$,
the heuristic argument of Ruelle \cite{Ru1ph}
suggests to define the susceptibility function  ~as:
\begin{align*}
\Psi(z)&=\sum_{n=0}^\infty
\int z^n \LL_1(v \rho_0 )(y) ( \varphi \circ f^n)' (y)\, dy =\sum_{n=0}^\infty
\int z^n \LL^n_0\bigl (\LL_1(v\rho_0 ) \bigr )(x) \varphi' (x)\, dx\, .
\end{align*}
The analysis of the above expressions produces additional
difficulties, and will not be pursued here.
\end{remark}

Since $X\rho_0 \in BV$, Lemma \ref{sp} implies that
the power series $\Psi(z)$ extends to a holomorphic function in the open unit
disc, and in this disc we have 
$$
\Psi(z)=
\int (\id-z \LL_0)^{-1} (X  \rho_0) (x) \, \varphi' (x)\, dx \, .
$$

Recalling the jumps 
$s_n$ in the saltus term $\rho_{sal}$ for $\rho$ (see (\ref{dec0})),
the weighted total jump of $f$ defined in \cite{Ba} is:
\begin{align}\label{saltusid}
\JJ(f,X)&=\sum_{n=1}^{N_f} s_{n} X(c_{n}) \, .
\end{align}

In \cite{Ba}, we resummed the possibly
divergent series $\Psi(1)$ under the condition
$\JJ(f,X)=0$ (see Proposition~\ref{candidate}
below). We have the following simple but enlightening lemma:

\begin{lemma}\label{nice}
Assume that $f$ is 
a piecewise expanding $C^2$ unimodal map $f$,
and that $X : I \to \real$ is bounded.
Define $\alpha_{(0)}(c_1)$ by (\ref{aalpha}) for
$v=X\circ f$. Then 
$$\JJ(f,X)=
s_1(X(c_1)-\alpha_{(0)}(c_1))\, .
$$
\end{lemma}

Since  $s_1<0$,  the lemma
implies  $\JJ(f,X)=0$ if and only if $\alpha_{(0)}(c_1)=X(c_1)$, i.e.,
if and only if $X \circ f$ is horizontal for $f$.

\begin{proof}
If $c$ is neither periodic nor preperiodic, then
 $s_k = f'(c_k) s_{k+1}$ for $k\ge1$, and thus 
\begin{equation}
  \JJ(f,X)=
s_1(X(c_1)-\alpha_{(0)}(c_1))=s_1  \sum_{j\ge 0} \frac{X(f^{j}(c_1))}{(f^{j})'(c_1)}
\end{equation}
(see  \cite[Rem. 4.5]{Ba}). 
The case of periodic $c$ is similar using $s_k = f'(c_k) s_{k+1}$ for $1\le k \le n_1-1$
and $M_f=n_1$.

If $c$ is preperiodic,
using $s_k = f'(c_k) s_{k+1}$ for $1\le k \le n_0+n_1-2$,
$k\ne n_0-1$, and 
$$s_{n_0}=
\frac{s_{n_0-1}} {f' (c_{n_0-1})} +
\frac{s_{n_0+n_1-1} }{f' (c_{n_0+n_1-1})}=
\frac{s_{n_0-1}} {f' (c_{n_0-1})} +
\frac{s_{n_0} }{(f^{n_1})' (c_{n_0})}
\, ,
$$
which implies  $(1-(f^{n_1})'(c_{n_0}))s_{n_0}= s_{n_0-1}/(f'(c_{n_0-1}))$
and thus 
$$s_{n_0+j}=\frac{s_1}{(f^{n_0+j-1})'(c_1)}
\frac{1}{1-1/(f^{n_1})'(c_{n_0})}\, , \quad 0\le j \le n_1-1\, ,
$$
we get
\begin{align*}
\JJ(f,X)&=
s_1 \biggr (\sum_{n=0}^{n_0-2}
 \frac{ X(f^{n}c_{1}) }{(f^{n})'(c_1)}+ \sum_{j=0}^{n_1-1}
 \frac{X(f^{n_0+j-1}(c_1))} {(f^{n_0+j-1})'(c_1)}
\frac{1}{1-1/(f^{n_1})'(c_{n_0})} 
\biggl )\\
&=s_1(X(c_1)-\alpha_{(0)}(c_1))\, .
\end{align*}
\end{proof}

We next recall the candidate $\Psi_1$ for the
derivative of $t\mapsto \RR(t)$ from Ruelle's conjecture as interpreted in \cite{Ba}.
Note that if  $X\in C^2(f(I))$  satisfies $X(a)=0$
then the function $\widetilde X$
defined by $\widetilde X(x):=X(x)$ for $x\ge a$
and $\widetilde X(x):=0$ for $x\le a$ is such that $\widetilde X'$ is  of bounded 
variation, and $\widetilde X' \tilde \rho$ is supported in $[a,b]$ for any $\tilde \rho$ 
supported in $(-\infty, b]$.
Recall $M_f$ from (\ref{Mf}).
Then, by Proposition~\ref{decc} and the properties
of $s_k$ from the proof of Lemma~\ref{nice},
putting together \cite[Lemma 4.1, Proposition 4.4, Theorem 5.2]{Ba}
gives
\footnote{Theorem 5.2 in \cite{Ba} also holds if $c$ is periodic,
with a similar proof.}:

\begin{proposition}\label{candidate}
Let $f$ be  a mixing
piecewise expanding $C^3$ unimodal map. 
Let $X \in C^2(f(I))$ satisfy $X(a)=0$ and
$\JJ(f,X)=0$.
For $\varphi \in C^1([a,b])$ and
$|z|< 1$:
\begin{align}\label{chsusceptibility}
\Psi(z)
&=-\sum_{j=1}^{\infty} \varphi(c_j)  \sum_{k=1}^{\min(j,M_f)} z^{j-k}  
  \frac {s_1 X(c_k)}{(f^{k-1})'(c_1)}\\
\nonumber &\qquad\qquad\qquad-
\int (\id- z\LL_1)^{-1} (X' \rho_{sal}+(X\rho_{reg})') \varphi \, dx  \, .
\end{align}
The second term in (\ref{chsusceptibility}) extends to a holomorphic function in the
open disc
of radius  $\lambda_0^{-1}$.
If $c$ is periodic or preperiodic
then the first term of (\ref{chsusceptibility})
is a rational function which is holomorphic at $z=1$. 

In addition, 
the following is a well-defined complex number
\begin{align}\label{psi1}
\Psi_1
&=-\sum_{j=1}^{M_f} \varphi(c_j)  \sum_{k=1}^j    \frac {s_1 X(c_k)}{(f^{k-1})'(c_1)}
-
\int (\id- \LL_1)^{-1} (X' \rho_{sal}+(X\rho_{reg})') \varphi \, dx \, .
\end{align}
\end{proposition}

Note that if $\JJ(f,X)=0$ (a codimension one condition on $X$) then
$\Psi_1=\Psi_1(\varphi)$ is well-defined even if $\varphi$ is only continuous.

\smallskip
We have the following simpler expression for the first term
of $\Psi_1$:

\begin{lemma}\label{reallynice}
Let $f$ be  a mixing
piecewise expanding $C^3$ unimodal map. 
Let $X \in C^2(f(I))$ satisfy $X(a)=0$ and
$\JJ(f,X)=0$, and let $\varphi \in C^1([a,b])$.
Then, setting
$\alpha=\alpha_{(0)}$ from (\ref{aalpha}) for $f$ and $v=X \circ f$,
\begin{equation}\label{nicer}
\Psi_1=
-\int \alpha  \varphi \, \rho'_{sal}-
\int (\id- \LL_1)^{-1} (X' \rho_{sal}+(X\rho_{reg})') \varphi \, dx
\, .
\end{equation}
\end{lemma}

\begin{proof}
By Lemma~\ref{nice} $X(c_1)=\alpha(c_1)$. Thus, by (\ref{forl})
the first term of $\Psi_1$ from (\ref{psi1}) may be rewritten as a Stieltjes integral
($\alpha$ is continuous by Corollary \ref{horiz})
\begin{align}\label{ST1}
-\nonumber s_1 \sum_{j=1}^{M_f}  \varphi(c_j) 
\biggl (X(c_1)-\alpha(c_1)+\frac{\alpha(c_j)}{(f^{j-1})'(c_1)} \biggr )
&= -s_1 \sum_{j=1}^{M_f} \varphi(c_j) \frac{\alpha(c_j)}{(f^{j-1})'(c_1)}\\
 &=-\int \alpha  \varphi \, \rho'_{sal}\, .
\end{align}
\end{proof}

In fact, $\Psi_1$ is well-defined only if $\JJ(f,X)=0$:

\begin{proposition}\label{nono}
Let $f$ be a mixing piecewise expanding
$C^3$ unimodal map $f$, let $X \in C^2(f(I))$
satisfy $X(a)=0$. 
For every $\varphi \in C^0([a,b])$ the following series converges
$$- \sum_{j=1}^\infty
\int  \LL_1^{j} ((X\rho_{reg})')(x) \varphi (x)\, dx\, .
$$ 
If $\JJ(f,X)\ne 0$ then  $\Psi_1$
is not well-defined, in the following sense:
There exists $\varphi \in C^\infty([a,b])$ so that,
on the one hand, both series below
diverge
\begin{equation}\label{nor}
-\sum_{j=1}^{\infty} \varphi(c_j)  \sum_{k=1}^{\min(j,M_f)}     \frac {s_1 X(c_k)}{(f^{k-1})'(c_1)}
- \sum_{j=1}^\infty
\int  \LL_1^{j} (X' \rho_{sal})(x) \varphi (x)\, dx \, ,
\end{equation}
and on the other hand, the following series diverges
\begin{equation}\label{allt}
-\sum_{j=1}^{\infty} \biggl (\varphi(c_j)  \sum_{k=1}^{\min(j,M_f)}    \frac {s_1 X(c_k)}{(f^{k-1})'(c_1)}
+
\int  \LL_1^{j} (X' \rho_{sal})(x) \varphi (x)\, dx \biggr )\, .
\end{equation}
\end{proposition}

\begin{proof}
Since $\int ((X\rho_{reg})')(x)\, dx=0$, the proof of \cite[Proposition 4.4]{Ba},
implies
$$\bigl |-\int  \LL_1^{j} ((X\rho_{reg})')(x) \varphi (x)\, dx \bigr |\le C \tau^j\, ,
$$
which
gives the first claim.

To fix ideas assume that $\JJ(f,X)> 0$. Recalling Lemma~\ref{nice}, note that
if $c$ is not periodic, then for each $j$
\begin{equation}\label{forl}
\sum_{k=1}^j s_k X(c_k)= X(c_1)-\alpha_{(0)}(c_1)+
\frac{\alpha_{(0)}(c_j)}{(f^{j-1})'(c_1)}=\JJ(X,f)+
\frac{\alpha_{(0)}(c_j)}{(f^{j-1})'(c_1)}\, .
\end{equation}
By the proof of \cite[Proposition 4.4]{Ba},
$$
\bigl |-\int  \LL_1^{j} (X' \rho_{sal})(x) \varphi (x)\, dx
- \JJ(f,X)\int \varphi\rho_0\, dx \bigr |\le C \tau^j
\, ,
$$
thus if  $\int \varphi \rho_0\, dx > 0$
then the second term in (\ref{nor}) diverges to $+\infty$.
If $c$ is not periodic and, in addition, $\inf_j \varphi(c_j)> \int \varphi \rho_0\, dx>0$
then the first term diverges
to $-\infty$ (use (\ref{forl})). 
Finally, for the same $\varphi$, if $c$ is not
periodic then (\ref{allt}) is $\JJ(f,X) \sum_j (-\varphi(c_j)+\int \varphi\rho_0\, dx)$,
which clearly diverges to $-\infty$.
The case of periodic $c$ is similar.
\end{proof}

We end this section by discussing the relation between $\Psi(z)$ and
$\Psi_1$ when $\JJ(f,X)=0$:
If $c$ is preperiodic or periodic, $\Psi_1$ is just the value at $1$ of the 
holomorphic extension
of $\Psi(z)$, and we have $\Psi_1=\lim_{z \to 1} \Psi(z)$.
If $c$ is neither periodic nor
preperiodic we do not know if
the resummation $\Psi_1$
for the possibly divergent series $\Psi(1)$
is always Abelian, i.e., if $\Psi_1=\lim_{z\in (0,1),z\to 1}
 \Psi(z)$, but we have the following sufficient codimension-two
condition on $X$ ensuring abelianity:

\begin{proposition}\label{abelian}
Let $f$ be  a mixing
piecewise expanding $C^3$ unimodal map. 
Let $X \in C^2(f(I))$ satisfy $X(a)=0$,
$\JJ(f,X)=0$, and, in addition,
\begin{equation} \label{a2}  
\sum_{j=1}^{\infty}  \frac{j \ X(f^{j}(c_1))}{(f^ j)'(c_1))}=0
\end{equation}
then $\Psi_1=\lim_{z\in (0,1),z\to 1} \Psi(z)$.
\end{proposition}

\begin{proof}
We may assume that the critical point $c$ is not periodic, so that 
the following formal Laurent series is well-defined for $\ell\ge 1$:
$$
\alpha(c_\ell,z)= -\sum_{j= 1}^ \infty \frac{X(f^j(c_\ell))}{z^j (f^ j)'(c_\ell)}
\, .
$$
Clearly,
$z\mapsto \alpha(c_1,z)$
is analytic in $\{ z \in \complex \mid |z| \min |f'|> 1\}$. We have
\begin{equation}\label{derivative}
\partial_z \alpha(c_1,z)|_{z=1}= \sum_{j=1}^ \infty \frac{j \ X(f^{j}(c_1))}{(f^ j)'(c_1)}
\, .
\end{equation}
Note for further use that if 
$X(c_1) =\alpha(c_1,1)$ (which is equivalent to $\JJ(f,X)=0$ by Lemma~ \ref{nice})
and if (\ref{a2}) holds,
then (\ref{derivative}) implies
\begin{equation}\label{a3}
\alpha(c_1,z) = \alpha(c_1,1)+ O(|1-z|^2)\, .
\end{equation}

Now, using $\alpha(c_1,z)$, we may rewrite the coefficient of
$\varphi(c_j)$ in the first term of $\Psi(z)$ from Proposition~ \ref{candidate} as
\begin{align*}
s_1  z^{j-1}&
\sum_{k=1}^{j}\frac{X(c_k)} {z^{k-1}(f^{k-1})'(c_1)}\\ &=
 s_1  z^{j-1}\Big(   X(c_1) -\alpha(c_1,z) -   
\frac{1}{z^{j-1} (f^{j-1})' (c_1)}\sum_{m=1}^{\infty} \frac{X(c_{j+m})} {z^{m}(f^{m})'(c_{j})} 
\Big)  \\
 &=
 s_1  z^{j-1} \big( X(c_1) -\alpha(c_1,z)\big) -   
s_1 \frac{\alpha(c_j,z)}{(f^{j-1})' (c_1)} 
\, .
\end{align*}
Consequently, if $(\min |f'|)^ {-1} < |z| < 1$,
the first term of $\Psi(z)$ can be written as
\begin{align*}
\sum_{j=1}^{\infty}& \varphi(c_j) \sum_{k=1}^{j} z^{j-k}\frac{s_1X(c_k)}{(f^ {k-1})'(c_1)}\\
&= 
\sum_{j=1}^{\infty} \varphi(c_j) \Big[ s_1  z^{j-1} \big( X(c_1) -\alpha(c_1,z)\big) 
-   s_1 \frac{\alpha(c_j,z)}{(f^{j-1})' (c_1)}\Big] \\
&=s_1 \Big[   \big( X(c_1) -\alpha(c_1,z)\big) \sum_{j=1}^{\infty} \varphi(c_j) z^{j-1} 
-  \sum_{j=1}^{\infty}  \frac{\varphi(c_j)\ \alpha(c_j,z)}{(f^{j-1})' (c_1)}\Big] \, .
\end{align*}

It is easy to see that
$$\lim_{|z|< 1, z\rightarrow 1}   \sum_{j=1}^{\infty}  
\frac{\varphi(c_j)\ \alpha(c_j,z)}{(f^{j-1})' (c_1)}=  
\sum_{j=1}^{\infty}  \frac{\varphi(c_j)\ \alpha(c_j,1)}{(f^{j-1})' (c_1)}\, .$$
Note also that if $|z|< 1$ then 
$|\sum_{j=1}^{\infty} \varphi(c_j) z^{j-1} |\leq \frac{ \sup |\varphi|}{1-|z|}$.

Finally, (\ref{a3}) implies
$$|\big( X(c_1) -\alpha(c_1,z)\big) \sum_{j=1}^{\infty} \varphi(c_j) z^{j-1}|=
|\big( \alpha(c_1,1) -\alpha(c_1,z)\big) 
\sum_{j=1}^{\infty} \varphi(c_j) z^{j-1}|\leq C|z-1|
\, .
$$
Putting together the above estimates, we find using (\ref{forl})
\begin{align*}
\lim_{z\rightarrow 1^{-}}  \sum_{j=1}^{\infty} \varphi(c_j) 
\sum_{k=1}^{j} z^{j-k}\frac{s_1X(c_k)}{(f^ {k-1})'(c_1)}&=
-  s_1\sum_{j=1}^{\infty}  \frac{\varphi(c_j)
\alpha(c_j,1)}{(f^{j-1})' (c_1)}\\
&=
\sum_{j=1}^\infty \varphi(c_j) \sum_{k=1}^ j \frac{s_1 X(c_k)}{(f^ {k-1})'(c_1)}\, ,
\end{align*}
which immediately gives the claim.
\end{proof}

\section{Proof of the main theorem}
\label{four}

If $f_t$ is a $C^{2,2}$ perturbation of a  mixing
piecewise expanding $C^2$ unimodal map $f$ tangent to its topological class,
then Corollary~\ref{horiz} gives that the infinitesimal deformation
$v$ is horizontal. If $v=X \circ f$, Lemma~\ref{nice} thus implies that
$\JJ(f,X)=0$. Therefore, if $X \in C^2(f(I))$,
a candidate $\Psi_1$ for the derivative
is defined by Proposition ~\ref{candidate} and Lemma~\ref{reallynice}.
Our main result can now be stated:

\begin{theorem}\label{formula}
Let $f_t$ be a 
$C^{2,3}$ perturbation
of a  mixing piecewise expanding $C^3$ unimodal
map $f$ with infinitesimal
deformation $v= X\circ f$ such that $X \in C^2(f(I))$.
If $f_0$ is good and $f_t$ is tangent to
its topological class, or if $f_t=\tilde f_t$
lies in the topological class of $f_0$,
then $t\mapsto \rho_t\, dx$ from
$(-\epsilon,\epsilon)$  to Radon measures is differentiable at $0$,
and 
$$
\partial_t (\rho_t \, dx)|_{t=0}=- \alpha  \rho'_{sal}- (\id- \LL_1)^{-1} (X' \rho_{sal}+(X\rho_{reg})') \, .
$$
In particular, for any $\hat \varphi \in C^{0}([a,b])$, the map 
$\RR(t)=\int \hat \varphi \, \rho_t\, dx$ is differentiable
at $t=0$, and 
$\RR'(0)=\Psi_1(\hat \varphi)$.
\end{theorem}

\begin{remark}\label{nonho}
See Theorem~\ref{nonlip} for necessity of the condition $\JJ(f,X)=0$
(which is equivalent to tangency to the topological class by Corollary ~\ref{horiz}).
\end{remark}

\begin{proof}[Proof of Theorem~ \ref{formula}]
Since 
$\tilde f_t= f_t$ if $f$ is not good,  we may assume  without loss of
generality by Proposition~\ref{fromKL}
that $\tilde f_t=f_t=h_t\circ f \circ h_t^{-1}$
for all $t$. Also, since each $\rho_t$ is a probability measure,
we may restrict to continuous functions $\hat \varphi$ so that $\int \hat \varphi\, d\rho_0=0$.
The proof will then be divided in three steps.

\smallskip
{\bf Step 1: Perturbation theory via resolvents.}

Recall the spaces $\BB_{t}=\Gamma_t(\widehat \BB)$ from Subsection~\ref{3.2},
for a fixed $\eta >0$,
and define  linear isometries $G_t = \Gamma_0 \circ \Gamma_t^{-1}: \BB_{t} \to \BB_{0}$.
We decompose
\begin{equation}\label{dec1}
\rho_t-\rho_0 = (G_t (\rho_t )- \rho_0) + (\rho_t - G_t(\rho_t))\, .
\end{equation}
The second term may be analysed directly, noting that (as Radon measures)
$$
\lim_{t \to 0}\frac{\rho_{t }- G_t (\rho_{t})}{t}=
\lim_{t \to 0}\frac{\rho_{sal,t }- \rho_{sal,t} \circ h_t}{t}=
 -\sum_{k =1} ^{N_f}  \alpha(c_k) s_k \delta_{c_k} =
-\alpha \rho'_{sal} \, . 
$$
(We used that $c_{k,t}=h_t(c_k)$ implies $H_{c_k}=H_{c_{k,t}} \circ h_t$ and that
$\lim_{t\to 0}s_{1,t}=s_1$\footnote{For this claim (which implies 
$\lim_{t\to 0} s_{k,t}= s_k$ for each fixed $k$), use that $\lim_{t\to 0}
\int \varphi \rho_t \, dx =
\int \varphi \rho\, dx$ for all bounded $\varphi$:
Since $\lim_{t\to 0}c_{1,t}= c_1$,
and $\sup_t \| \rho_{reg,t}\|_{Lip} <\infty$, while
$|s_{k,t}|\le C \lambda^k$ uniformly in $t$, choosing for
$\varphi$ the characteristic function of a sufficiently small neighbourhood of $c_1$,
we get a contradiction if $s_{1,t} \not\to s_1$.} 
.)
To study the first term in (\ref{dec1}), set  
$$
\PP_t=G_t\circ \LL_{1,t} \circ G_t^{-1}\, ,
\quad \widehat \QQ_t=\widehat \QQ_t(z)= z-\PP_t\, .
$$
(Of course $\PP_0=\LL_1$ and $\widehat \QQ_0=z-\LL_1$.)
The operator $\PP_t$ on $\BB_{0}$ is conjugated to $\LL_{1,t}$ on
$\BB_{t}$ and therefore has the same spectrum. The fixed point 
of $\PP_t$ is $G_t(\rho_t)$ and
the fixed point of $\PP_t^*$ is
$\nu_t(\varphi)=\int G_t^{-1}(\varphi)\, dx$. We denote by $\widehat \Pi_t(\varphi)=G_t(\rho_t)\nu_t(\varphi)$ the corresponding spectral projector.
Our strategy will be to use, as in Proposition~\ref{fromKL},
$$
\widehat \QQ_t^{-1}- \widehat\QQ_0^{-1}= \widehat \QQ_t^{-1} (\PP_t -\PP_0)
 \widehat\QQ_0^{-1}\, ,
$$
in order to write $G_t(\rho_t) \nu_t(\varphi_0)-\rho_0 \int \varphi_0 \, dx$ 
as a difference of spectral projectors
applied to $\varphi_0 \in \widetilde \BB_{0}$, where 
$$\widetilde \BB_{0}= 
\{\varphi \in \BB_0\mid \varphi'_{reg} \in  \BB^{Lip}_0\}
\mbox { with
the  norm }
\|\varphi'_{reg} \|_{ \BB^{Lip}_0}+ \|\varphi \|_{\BB_0}\, .
$$
In fact, we do not need to perform the spectral analysis of $\LL_1$ on 
$\widetilde \BB_{0}$,
since we shall work exclusively with $\rho_0 \in \widetilde \BB_{0}$
(the fact that $\rho_{reg}' \in \BB_0^{Lip}$, i.e., that all discontinuities
of $\rho_{reg}'$ lie on the postcritical orbit,  that the
jump at $c_k$ is  $O(\lambda^{k})$, and that $(\rho_{reg})'_{reg}\in Lip$
 is an easy consequence of
the proof of \cite[Proposition 3.3]{Ba}, noting in particular
the uniform bound for $\Delta'_n(x)$ there -- see also (\ref{star}) and (\ref{rels'})).

Since $\int \rho_0 \, dx=1$,
noting that $\widehat \QQ_0^{-1} (\rho_0)= \rho_0/(z-1)$, we
find
\begin{align}\label{idea}
G_t(\rho_t)\nu_t(\rho_0)-\rho_0 &=
- \frac{1}{2 i \pi} \oint  \frac{\widehat \QQ_t^{-1}(z)} {z-1} (\PP_t -\PP_0) (\rho_0)\, dz\\
\nonumber &=
(\id -\PP_t)^{-1} (\id-\widehat \Pi_t)(\PP_t-\PP_0) (\rho_0) \, ,
\end{align}
where the  contour is a circle centered at $1$,
outside of the disc of radius $\tau$.

We shall also use the following norms on $\BB_0$,
for $j \ge 0$
$$
|\varphi|_{weak,j} = \frac{\| \varphi_{reg} \|_{L^1(Leb)} }{2}
+\frac{ \max \{| \varphi_{reg} (y)| \mid y \in \cup_{0 \le \ell \le j} f^{-\ell}(c) \}}{2}
+ | \Gamma^{-1}(\varphi_{sal}) |_{\eta}\,  .
$$
We have 
$ |\varphi|_{weak, j} \le   \|\varphi\|_{\BB_0}$ for all $j \ge 0$.
It   is not  difficult to see
by adapting the estimates in Subsection~\ref{3.2} that there exist $\epsilon >0$ and $C\ge 1$ 
so that,  for  all  $|t|\le \epsilon$
all $j$, $\ell$, all  $\varphi \in \BB_0$,
\begin{equation}\label{hand0}
|\PP_t^j(\varphi)|_{weak, \ell}  \le C |\varphi|_{weak, \ell+j }
\, , \quad
\|\PP_t^j (\varphi )\| \le C \lambda^{j} \|\varphi\|
+C |\varphi|_{weak, j} \, .
\end{equation}
(Uniformity in $t$ of the constant $C$ in the Lasota-Yorke estimate
 follows from the fact that  $f$ is good. The reason why 
$\sup_{\ell \le j} |\varphi_{reg}(f^{-\ell}(c))|$
appears in the weak norm is to take into account 
the compact operators $\KK_0(\LL_1^j)$ from
the decomposition in \S~\ref{3.2}.)
We shall see in Step~3  that for any fixed $j\ge 0$
there is a modulus of continuity $\delta_j (t)\ge 0$ 
(i.e., $\limsup_{t\to 0} \delta_j(t)=0$)
so that for each $\varphi \in \BB_0$ 
\begin{equation}\label{missarg'}
|\PP_t(\varphi)-\PP_0(\varphi)|_{weak, j }\le \delta_j(t)  \|\varphi\|_{\BB_0} \, .
\end{equation}
Therefore, the proof of \cite[Theorem 1]{KL} 
(see Appendix~\ref{newKL}) gives $\epsilon_0>0$ so that
\begin{equation}\label{newKL'}
A_{\epsilon_0}:=
\sup_{|t|< \epsilon_0}\|(\id -\PP_t)^{-1}(\id-\widehat \Pi_t)\|_{\BB_0}< \infty \, .
\end{equation}
Beware that it is not clear
whether $|(\id -\PP_t)^{-1}(\id-\widehat \Pi_t)(\varphi)- (\id -\PP_0)^{-1}(\varphi)|_{weak,0}$
tends to zero uniformly in $\|\varphi\|_{\BB_0}\le 1$ as
$t\to 0$.
This is why we next consider $\PP_t$ acting on
$\BB_0^{Lip}$: By \S ~\ref{3.2}, the essential
spectral radius is $\le \lambda$, and the spectrum  outside of the
disc of radius $\tau$ consists in the eigenvalue
$1$, with  projector $\widehat \Pi_t$. We introduce
a weak norm  on $\BB_0^{Lip}$:
$$
|\varphi|_{weak,\infty} = \| \varphi_{reg} \|_{L^\infty(Leb)} +
 | \Gamma^{-1}(\varphi_{sal}) |_{ \eta}\,  .
$$
Applying again the argument in \S ~\ref{3.2}, we see that (\ref{hand0}) holds
for $\ell=\infty$. Clearly, $ |\varphi|_{weak, j}\le |b-a| |\varphi|_{weak, \infty}$. 
In Step ~3, we shall find $\widetilde C \ge 1$
so that for each $\varphi \in \BB_0^{Lip}$
\begin{equation}\label{missarg}
|\PP_t(\varphi)-\PP_0(\varphi)|_{weak, \infty }\le \widetilde
C |t|  \|\varphi\|_{\BB_0^{Lip}} \, .
\end{equation}
Then,  setting
$$
\NN_t=(\id -\PP_t)^{-1}(\id-\widehat \Pi_t)- (\id -\PP_0)^{-1}(\id-\widehat \Pi_0)\, ,
$$
(\ref{hand0}) and (\ref{missarg})
imply by \cite[Theorem~1, Corollary ~1]{KL} that 
there are $\widehat C\ge 1$ and  $\xi >0$   so that for each 
$\varphi \in \BB_{0}^{Lip}$
\begin{equation}\label{limm}
|\NN_t(\varphi) |_{weak, \infty}
\le \widehat C |t|^\xi \|\varphi\|_{\BB_0^{Lip}}\, .
\end{equation}

If we knew that there existed $\DD \in \BB_0^{Lip}$ so that
\footnote{We emphasize that the norm in (\ref{almostder}) is  in $\BB_0$, and a priori not
in $\BB_0^{Lip}$.}
\begin{equation}\label{almostder}
\|\PP_t(\rho_0)-\PP_0(\rho_0) - t \DD \|_{\BB_0}= O(t^2) \, ,
\end{equation}
uniformly in small $t$ (this will be shown in Step~2),
then  (\ref{idea}) and (\ref{limm}) would give
\begin{equation}\label{claim1}
\partial_t (G_t (\rho_t) \nu_t(\rho_0))|_{t=0} =
 (\id-\LL_1)^{-1} (\id-\widehat \Pi_0) (\DD) \, ,
 \end{equation}
in $L^\infty(Leb)$: Indeed,
write $(\id-\PP_t)^{-1}(\id - \widehat \Pi_t)= \NN_t +(\id-\PP_0)^{-1}(\id - \widehat \Pi_0)$
and note that (\ref{newKL'})
implies
\begin{align}\label{lllast}
G_t(\rho_t)\nu_t(\rho_0)-\rho_0 &=
(\NN_t +(\id-\PP_0)^{-1}(\id - \widehat \Pi_0))(t \DD + O_{\BB_0}(t^2)) \\
\nonumber &= t  \NN_t(\DD) +t(\id-\PP_0)^{-1}(\id - \widehat \Pi_0)(\DD)
+ A_{\epsilon_0} O(t^2) \, .
\end{align}
Dividing by $t$ and letting $t\to 0$, (\ref{limm}) gives the claim (\ref{claim1}).

Note that $t \mapsto \nu_t(\rho_0)$ is differentiable at $0$: As
$\nu_t(\rho_0)=\int \rho_{sal}\circ h_t^{-1}\, dx+ \int \rho_{reg}\, dx$,
one easily sees that $\partial_t \nu_t(\rho_0)|_{t=0}=-\sum_{k= 1}^{N_f}
\alpha(c_k) s_k$.
Then, by the Leibniz formula,
\begin{equation}\label{consLeib}
\partial_t (G_t \rho_t)|_{t=0}=
\partial_t (G_t (\rho_t) \nu_t(\rho_0))|_{t=0} -\rho_0 \, \partial_t( \nu_t(\rho_0))|_{t=0}  \, .
\end{equation}
Since our test functions satisfy $\int \hat \varphi\, d\rho_0=0$, we can ignore
scalar multiples of $\rho_0$, and
it only remains to show (\ref{missarg'}), (\ref{missarg}),
and (\ref{almostder}) with
\begin{equation}\label{magic}
(\id -\widehat \Pi_0)(\DD )
=-X' \rho_0-X\rho_{reg}' \, .
\end{equation}

\smallskip

{\bf Step 2: Analysing the derivative of $t\mapsto \PP_t(\rho_0)$.}

In this step, we prove (\ref{almostder}) and (\ref{magic}).
By definition, for any $\varphi \in \BB_{0}$
\begin{align}\label{encore}
\PP_t( \varphi )&=(\LL_{1,t}(\varphi_{sal}\circ h_t^{-1}+\varphi_{reg}))_{sal} \circ h_t
+(\LL_{1,t}(\varphi_{sal}\circ h_t^{-1}+\varphi_{reg}))_{reg}  \, .
\end{align}

From now on, we assume that the postcritical orbit is infinite, to fix
ideas. (The case of finite postcritical orbit is similar.)
Recall (\ref{spart'}).
Noting that $c_k> c$ if and only
if $c_{k,t}=f^k_t(c)>c$,  and writing $\varphi_{sal}=\sum_k u_k H_{c_k}$, the contribution
to $\PP_t(\varphi)-\PP_0(\varphi)$ from the first
term in the right-hand-side of (\ref{encore}), i.e.,
$(\PP_t(\varphi))_{sal}-\PP_0(\varphi)_{sal}$, is just
\begin{align}\label{III}
\sum_{k =2}^{N_f} &u_{k-1} \biggl (\frac {1}{f'_t(c_{k-1,t})} - \frac{1}{f'(c_{k-1})}\biggr )
H_{c_k}\\
\nonumber &+
(\varphi_{reg}(c)+ \sum_{c_k > c} u_k) \biggr (\frac{1}{f'_t(c_-)} -  \frac{1}{ f'(c_-)} -
 \frac{1}{f'_t (c_+)} + \frac{1}{ f' (c_+)} \biggl )H_{c_1}\, .
\end{align}

Next,  we
find  by  (\ref{rpart}) that the derivative of  the second term 
$((\PP_t(\varphi))_{reg}-\PP_0(\varphi)_{reg})$ 
of (\ref{encore}), which is an atomless measure,  coincides with
\begin{align}\label{lastone}
&(\LL_{1,t} (\varphi_{reg} ))'|_{(a,c_{1,t})}- (\LL_1 (\varphi_{reg}))'|_{(a,c_1)}\\
 &\nonumber\qquad + \sum_{k= 2, c_{k-1}>c}^{N_f} u_{k-1}
\bigl ((\LL_{1,t}(H_{c_{k-1,t}}))'|_{(c_{k,t},c_{1,t})} -
(\LL_1(H_{c_{k-1}}))'|_{(c_k,c_1)} \bigr )\\
&\nonumber\qquad + \sum_{k= 2}^{N_f} u_{k-1}
\bigl ((\LL_{1,t}(H_{c_{k-1,t}}))'|_{(a,c_{k,t})} -
(\LL_1(H_{c_{k-1}}))'|_{(a,c_k)} \bigr ) \, .
\end{align}

Put $\varphi=\rho_0$, and consider first (\ref{III}).
Note that $c_{k,t}=h_t(c_k)$. Write
$$
\frac{1}{f'_t(h_t(w))}-\frac{1}{f'(w)}=
\frac{f'(w)-f'_t(h_t(w))}{f'_t(h_t(w))f'(w)}\, ,
$$
and decompose
$f'(w)-f'_t(h_t(w))=f'(w)-f'_t(w)+ f'_t(w)-f'_t(h_t(w))$,
with $f'(w)-f'_t(w)=-t X'(f(w)) f'(w)+O(t^2)$,
and $f'_t(w)-f'_t(h_t(w))=-t f''_t(w)\alpha(w)+O(t^2)$.
Thus, we find, by using $(\LL_1(\rho))_{sal} =\rho_{sal}$
and  (\ref{tceeq}), that
\begin{align}\label{glu}
\nonumber \lim_{t\to0}&\frac
{ (\PP_t(\rho_0))_{sal}-(\rho_0)_{sal}}{t}=
-\sum_{k=1}^{N_f}  X'(c_k) s_k H_{c_k}
-\sum_{k= 2}^{N_f}\frac{\alpha(c_{k-1}) s_{k-1} f''(c_{k-1})}{(f'(c_{k-1}))^2} H_{c_k}\\
\nonumber &\qquad
=-\sum_{k=1}^{N_f}  X'(c_k) s_k H_{c_k}
+\sum_{k= 2}^{N_f}\frac{(X(c_k)-\alpha(c_{k})) 
s_{k-1} f''(c_{k-1})}{(f'(c_{k-1}))^3} H_{c_k}
\\
&\qquad=-(X' \rho)_{sal}+  \sum_{k=1}^{N_f} (X(c_k)-\alpha(c_k))E_k H_{c_k} \, ,
\end{align}
where  we used    $X(c_1)=\alpha(c_1)$ with
(the choice of $E_1$
will become clear later on)
\begin{align}\label{ddef}
E_k &=
\frac{s_{k-1} 
f''(c_{k-1})}{(f'(c_{k-1}))^3} \, , \,\, k\ge 2\, , \\
\nonumber  E_1&= \biggl (-
 \frac{\rho_{reg} (c)f''(c_-)}{(f'(c_-))^3}
+  \frac{\rho_{reg} (c)f''(c_+)}{(f'(c_+))^3}\biggr ) \\
\nonumber
&\qquad\qquad\qquad\qquad+\sum_{k \ge 2, c_{k-1}>c} s_{k-1}
\biggl (\frac{f'' (c_-)}{(f'(c_-))^3} -\frac{f'' (c_+)}{(f'(c_+))^3} \biggr )\, .
\end{align}

It will turn out essential to study
$((\rho_{reg})')_{sal}=\sum_{k=1}^{N_k} s'_k H_{c_k}$.
If $x\in [a,c_1)$ is not along the critical orbit we have
\begin{equation}\label{star}
(\rho_{reg})'(x)=(\rho_0)'(x)=(\LL_1(\rho_0))'(x)
=\sum_{f(y)=x} \frac{(\rho_{reg})'(y)}{|f'(y)|f'(y)}-
\frac{\rho_0(y) f''(y)}{|f'(y)| (f'(y))^2}\, .
\end{equation}
(We used $(\rho_{reg})'(y)=(\rho_0)'(y)$ if $y$ is not along the postcritical orbit.)
Taking the difference between
$(\rho_{reg})'(x)$ for $x\uparrow c_k$ and $x\downarrow c_k$, and
recalling $E_k$ from (\ref{ddef}),
we easily get from the previous identity that 
\footnote{If $c$ is periodic then $(\rho_{reg})'(c)$ may be undefined,
but $(\rho_{reg})'(c_\pm)$ are both defined.} 
\begin{equation}\label{rels'}
s'_k=E'_k-E_k\, , \mbox{ with }
E'_k=\frac{s'_{k-1}}{(f'(c_{k-1})^2)} \,, \, k\ge 2\, , \, \,
 E'_1=-\frac{(\rho_{reg})' (c)}{(f'(c_-))^2}
+ \frac{(\rho_{reg})' (c)}{(f'(c_+))^2}  \,  .
\end{equation}

We now consider $\lim_{t\to 0}\frac{1}{t}((\PP_t(\rho_0))_{reg}-(\rho_0)_{reg})'$.
We get two sorts of contributions to (\ref{lastone}):
For  
\begin{equation}\label{sstar}
x\in [\min(c_k, c_{k,t}), \max (c_k, c_{k,t})]
\mbox{ or }
x\in [\min(c_k, f_t(c_{k-1})), \max(c_k, f_t(c_{k-1}))]\, ,
\end{equation}
an atom may appear at $c_k$ in the limit,  we call such $x$
singular points.
For the other values of $x$, which we call
the regular points, the limit will be a function.
Recalling (\ref{ddef}) and (\ref{rels'}), 
we claim that the  contribution of the singular points to
$\lim_{t\to 0} ((\PP_t(\rho_0))_{reg}-(\rho_0)_{reg})|_{(a,b)})'/t$ is 
\begin{align}\label{gli}
\sum_{k=1}^{N_f}( \alpha(c_k) E_k 
-X(c_{k})E'_k )\delta_{c_k} \, .
\end{align}
Indeed,  if $ k\ge 2$ and $c_{k,t}<c_k$ and $c_{k-1}<c$, we must consider the Radon measure 
\begin{align*}
\varphi \mapsto& -\frac{s_{k-1}}{t} \int_{c_{k,t}}^{c_k}
\frac{f'' (\psi_-(x))}{(f'(\psi_-(x)))^3} \varphi(x) \, dx
= \alpha(c_{k})s_{k-1} \frac{f'' (c_{k-1})}{(f'(c_{k-1}))^3} 
\varphi(c_k)+  O(t)\, ,
\end{align*}
coming from $-(\LL_1(H_{c_{k-1}}))'$ (we used  $h_t(c_{k})=c_{k,t}$).
If $ k\ge 2$, $c_{k,t}<c_k$, and $c_{k-1}> c$, we must consider the Radon measure 
\begin{align*}
\varphi \mapsto& -\frac{s_{k-1}}{t} \int_{c_{k,t}}^{c_k}
\frac{f''_t (\psi_{t,+}(x))}{(f'_t(\psi_{t,+}(x)))^3} \varphi(x) \, dx
= \alpha(c_{k}) s_{k-1}\frac{f'' (c_{k-1})}{(f'(c_{k-1}))^3} 
\varphi(c_k)+  O(t) ,
\end{align*}
from  $(\LL_{1,t}(H_{c_{k-1,t}}))'-(\LL_1(H_{c_{k-1}}))'$
(the corresponding term for the branches $\psi_-$ and $\psi_{t,-}$ vanishes
in the limit).
For $k=1$ and $c_{1,t}<c_1$ we must  consider the three contributions given by,
firstly,
\begin{align*}
\varphi \mapsto&- \frac{1}{t} \int_{c_{1,t}}^{c_1}
\frac{(\rho_{reg})' (\psi_-(x))}{(f'(\psi_-(x)))^2} \varphi(x) \, dx
 = \alpha(c_1)
\frac{(\rho_{reg})' (c)}{(f'(c_-))^2} 
\varphi(c_1)+  O(t)\, ,
\end{align*}
(recall also that $c_{1,t}=h_t(c_1)$ and  $\alpha(c_1)=X(c_1)$), secondly,
\begin{align*}
\varphi \mapsto& \frac{1}{t} \int_{c_{1,t}}^{c_1}
\frac{\rho_{reg} (\psi_-(x))f''(\psi_-(x))}{(f'(\psi_-(x)))^3} \varphi(x) \, dx
 = \alpha(c_1)
\frac{-\rho_{reg} (c)f''(c_-)}{(f'(c_-))^3} 
\varphi(c_1)+  O(t)\, ,
\end{align*}
and thirdly, by the sum over those $j\ge 2$ so that
$c_{j-1} > c$
of
\begin{align*}
\varphi \mapsto&- \frac{s_{j-1}}{t} \int_{c_{1,t}}^{c_1}
\frac{f'' (\psi_-(x))}{(f'(\psi_-(x)))^3} \varphi(x) \, dx
=   \alpha(c_1) s_{j-1}
\frac{f'' (c_-)}{(f'(c_-))^3} 
\varphi(c_1)+  O(t)\, ,
\end{align*}
as well as the corresponding three contributions for $\psi_+$.
The  cases $c_{k,t} > c_k$ are similar.
For $k \ge 2$, we must also deal with
the jump terms  from  $(\LL_{1,t}(\rho_{reg}))'-(\LL_1(\rho_{reg}))'$
(one at $f_t(c_{k-1})$ the other at $c_k$), which give,
using $f_t(c_{k-1})-f(c_{k-1})=tX(c_k)+O(t^2)$:
\begin{align*}
\varphi \mapsto&\frac{1}{t} \int_{f_t(c_{k-1})}^{c_k}
\frac{s'_{k-1}}{(f'(c_{k-1}))^2}\varphi(x) \, dx
= -X(c_k)\frac{s'_{k-1}}{(f'(c_{k-1}))^2}\varphi(c_k)+O(t) \, .
\end{align*}

We move to the regular points: For small $t$,  let
$k_t \ge 2$ be so that $\sum_{k \ge k_t} |s_{k-1}| \le  t^2$
(clearly, $k_t=O(\ln|t|)$), and  take $I_t$ to be the union
of the $O(k_t)$ intervals of singular points associated to
$k \le k_t$ via (\ref{sstar}) 
(in particular, the Lebesgue measure of $I_t$ is an $O(t\ln|t|)$). 
We have by definition
\begin{equation}\label{imp}
\| (\PP_t(\rho_0))_{reg}-(\rho_0)_{reg} -
(\LL_{1,t}(\rho_0)- \LL_1(\rho_0))_{reg}\|_{\BB_0(I \setminus I_t)} = O(t^2) \, ,
\end{equation}
where $\|\phi_{reg} \|_{\BB_0(I \setminus I_t)}$ is the norm  of Radon 
measure $(\phi_{reg})'$ on the metric set $I \setminus I_t$.
(For this, we use that 
$
\sum_{k\ge k_t} |s_{k-1}|
\| \LL_{1,t}(H_{c_{k-1,t}})- \LL_{1,t}(H_{c_{k-1}})\|_{\BB_0}=O(t^2)\, ,
$
and  that $\LL_{1,t}(H_{c_{k-1,t}})(x)- \LL_{1,t}(H_{c_{k-1}})(x)=0$ for 
$k \le k_t$ and $x\notin I_t$.)
The  contribution (\ref{gli}) takes care of 
$\| (\PP_t(\rho_0))_{reg}-(\rho_0)_{reg}\|_{\BB_0(I_t)}$ 
(note that $\sum_{k\ge k_t} |\alpha(c_k) E_k |
+|X(c_{k})E'_k |=O(t^2)$) so that we may concentrate on
$(\LL_{1,t}(\rho_0)- \LL_1(\rho_0))_{reg}$ on $I \setminus I_t$.

Note that
\begin{equation}\label{bothtypes}
f^{-1}(x)-f^{-1}_t(x)= t\frac{X(x)}{f'(f^{-1}(x))}+ O(t^2)\, ,
\end{equation}
where we choose the same inverse branch for $f_t$ and $f$.
It follows that 
\begin{align*}
& \frac{\varphi(f^{-1}_t(x))}{|f'_t(f^{-1}_t(x))|}- \frac{\varphi(f^{-1}(x))}{|f'(f^{-1}(x))|}
 =-t X'(x)\frac{\varphi(f^{-1}(x))}{|f'(f^{-1}(x))|} \\
 &\qquad\quad- t X(x)
 \biggl ( \frac{\varphi'(f^{-1}(x))}{ f'(f^{-1}(x))  | f'(f^{-1}(x))|} +
 \frac{\varphi(f^{-1}(x)) f''(f^{-1}(x))}{(f'(f^{-1}(x)))^2| f'(f^{-1}(x))|}\biggr )
 +O(t^2)\, ,
\end{align*}
if $\varphi$ is $C^{1+Lip}$ at $f^{-1}(x)$,
which gives, after summing over the two inverse branches,
\begin{equation}\label{uselater}
-tX' (x)\LL_1(\varphi)(x)- t X(x) (\LL_1(\varphi))'(x) +O(t^2) \, .
\end{equation}
Therefore, if $x \notin I_t$, and 
$x \ne c_k$ and $x\ne c_{k,t}$ for all $k\ge 1$, we have,
decomposing $\rho_0=\rho_{reg}+ \sum_k s_k H_{c_k}$, 
\begin{align}\label{usenow}
\nonumber (\LL_{1,t}(\rho_0)-\LL_1(\rho_0))_{reg} (x)&=
-t(X' \rho_0-X(\rho_0)')_{reg}(x) +O(t^2)\\
&=
-t(X'\rho_0)_{reg}(x)-t(X(\rho_{reg})')_{reg}(x) +O(t^2)\, .
\end{align}
(The  $O(t^2)$ term is in $\BB_0$, not $\BB_0^{Lip}$.)
By continuity, (\ref{usenow}) holds for all $x \notin I_t$.

The regular contribution to $\lim_{t\to 0}
{ \bigl ((\PP_t(\rho_0))_{reg}-(\rho_0)_{reg} \bigr ) }/{t}$ is thus
\begin{align}\label{gla}
-\bigl (X' \rho_0 - (X' \rho_0)_{sal}\bigr )
-\bigl (X( \rho_{reg})'- (X (\rho_{reg})')_{sal} \bigr )\, .
\end{align}
All together, we find from (\ref{glu}--\ref{gli}--\ref{gla}) and (\ref{rels'})
(differentiating in $\BB_{0}$)
$$
\partial_t (\PP_t(\rho_0)) |_{t=0} =-X' \rho_{sal}- X' \rho_{reg} 
- X( \rho_{reg})' \in \BB_0^{Lip}\, .
$$
This establishes (\ref{almostder}) and (\ref{magic})
(note that $\int X' \rho_{sal} + (X \rho_{reg})' \, dx=0$).

\smallskip

{\bf Step 3: Proving the weak norm bounds necessary for \cite{KL}.}

It remains  to prove the bounds (\ref{missarg'}) and  (\ref{missarg})
for $\PP_t(\varphi)-\PP_0(\varphi)$. We start with (\ref{missarg'}). For the
term corresponding to (\ref{III}), since
$\varphi$ is not necessarily a fixed point of $\LL_1$, we get in addition to
(\ref{glu}) a term
$$
(|\varphi_{reg}(c)|+ \sum_{c_k > c} |u_k |) O(t)
= O( t) |\varphi|_{weak, 0} \, .
$$
Next, consider (\ref{lastone}).
For the $L^1(Leb)$ norm
of $(\PP_t-\PP)_{reg}$, the singular contributions produce an $O(t \ln |t|)$ term:
Indeed, by (\ref{cpct}), up to an error $O(t)$ we may restrict to a finite
set of $c_k$s, where the cardinality of this finite
set is an $O(\ln |t|)$; for this finite set,
the total Lebesgue measure of the intervals of singular points
is an $O(t \ln |t|)$. For the regular contributions, although $\LL_1(\varphi)$ is
not equal to $\varphi$ in general, 
and $\varphi_{reg}$ is only continuous
and of bounded variation, we get an $O(t) \| \varphi\|_{\BB_0}$ contribution
to the $L^1(Leb)$ norm
of $(\PP_t-\PP)_{reg}$: Indeed, the only delicate terms are 
of the form
\begin{align*}
&\int h(y) (\varphi_{reg}(\psi_{+,t}(y))-\varphi_{reg}(\psi_+(y))) \, dy \, ,
\end{align*}
with $|h|\le \|f\|_{C^{1+Lip}}$, and similarly with $\psi_-$.
Now we exploit that if $\phi \in BV$ and $\Psi_t$ is $C^2$
with $|\Psi_t(x)-x|\le C|t|$ and $| \Psi_t'(x)-1| \le C|t|$ then
(use \cite[Lemma 11]{Ke} as in \cite[Lemma 13]{Ke})
$$
\int |\phi(y)-\phi(\Psi_t(y))|\, dy =O( t )\|\phi\|_{BV} \, .
$$

We must still bound
$|\PP_t(\varphi)_{reg}(y)-\PP_0(\varphi)_{reg}(y)|$ 
for $y \in \SS_j= \cup_{0\le \ell \le j} f^{-\ell}(c)$. 
We make no distinction
between regular and singular points here.
The contribution corresponding
to differences between derivatives of $f$  of $f_t$  gives $O(t)$.
Next,  $\varphi_{reg}$ is continuous by definition of $\BB_0$. Writing
$\tilde \delta_j(\cdot)$ for its worse modulus of continuity
on the finite set $\SS_j$ , we get since $|c_k-c_{k,t}|=O(t)$
that
$$
\sup_{y \in \SS_j}
|\PP_t(\varphi)_{reg}(y)-\PP_0(\varphi)_{reg}(y)|
=O (\tilde \delta_j(t) + |t|) \, .
$$

Finally,  (\ref{missarg}) can be proved by using the Lipschitz assumption
on $\varphi_{reg}$, to
simplify the  argument  
for (\ref{missarg'}): The uniform modulus of continuity $\delta(t)=O(t)$ 
of $\varphi_{reg}$
allows us to deal with the $L^\infty$ norm in $|\cdot|_{weak,\infty}$.
\end{proof}


\section{The derivative  in terms of the infinitesimal conjugacy $\alpha$}
\label{elegg}

Let $f_t$ be a  $C^{2,2}$ perturbation tangent to the topological
class of a  mixing piecewise expanding $C^2$ unimodal map.
We do not know whether $x\mapsto h_t(x)$ is quasisymmetric, as in the
smooth expanding case. Note however that in general
it is  {\it not} absolutely continuous (see \cite{MaM} for the nonuniformly
expanding case).
For similar reasons,
$\alpha=\partial_t h_t|_{t=0}$ is in general not absolutely
continuous.   In this section, we shall see that absolute continuity
of $\alpha$ is equivalent to a remarkable formula for $\Psi_1=\RR'(0)$ which
can be ``guessed" from the following easy lemma:

\begin{lemma}\label{funny}
Assume that $f_t$ is a $C^{2,2}$ perturbation
tangent to the topological class of  a piecewise expanding
$C^2$ unimodal map 
$f$, with infinitesimal perturbation
$v=X\circ f$. Then recalling
$\alpha= \partial_t h_t|_{t=0}$ from  Corollary ~\ref{horiz}, we have 
\begin{equation}
(\id - \LL_0) (\alpha \rho_0) =X \rho_0\, ,
\end{equation}
and
$
\sum_{k=0}^n \LL_0^k(X \rho_0)=\alpha \rho_0-\LL_0^{n+1}(\alpha \rho_0) 
$.
\end{lemma}

The lemma gives that the partial sum of order
$n$ for the series $\Psi(z)$ at $z=1$ is
$$
\sum_{k=0}^n
\int  \LL^k_0(X\rho_0 ) \varphi' \, dx=
 \int \varphi'\alpha \rho_0-\int \varphi' \LL_0^{n+1}(\alpha \rho_0) \, dx\, .
$$
We do not claim that $\int \varphi' \LL_0^{n+1}(\alpha \rho_0)\, dx$ 
converges as $n\to \infty$.

\begin{proof}
We know that
$X(y) = \alpha(y) - f'(\psi(y))\alpha(\psi(y))$
where $\psi$ is an arbitrary  inverse branch of $f$.
Multiply this by the positive number $\rho_0(\psi(y))/|f'(\psi(y))|$ and sum over 
inverse branches. Since   $\rho_0$ is the invariant density, the sum of these
positive numbers is $\rho_0(y)$, which gives the first claim.
A telescopic sum gives the second claim.
\end{proof}

\begin{theorem}\label{elegant}
Assume that $f_t$ is a $C^{2,3}$ perturbation tangent to the topological
class of a mixing piecewise expanding $C^3$ unimodal map $f$
with infinitesimal perturbation $v=X \circ f$
(in particular $\JJ(f,X)=0$) so that $X \in C^2(f(I))$.
If $\alpha=\partial_t h_t|_{t=0}$  is absolutely continuous then 
\begin{equation}\label{derr}
\Psi_1=\int \varphi' \alpha \rho_0 \, dx\, ,
\quad \forall \varphi \in C^1([a,b]) \, .
\end{equation}
Conversely, if (\ref{derr}) holds  then $\alpha\in BV^{(1)}$
(in particular, $\alpha$ is absolutely continuous).
\end{theorem}

Theorem~\ref{elegant} will easily imply:

\begin{corollary}[Derivative of the TCE]\label{cor1}
Under the assumptions of Theorem~\ref{elegant},
if $\alpha$ is absolutely continuous, then 
\begin{equation}\label{TTCE}
(-\id+\LL_1)
(\alpha' \rho_0 +\alpha (\rho_{reg})')
=X'\rho_0+X (\rho_{reg})'\, .
\end{equation}
\end{corollary}

Note that the proofs of Theorem~\ref{elegant} 
 and  Corollary~\ref{cor1}
use the results from \cite{Ba} (in particular Lemma 4.1, Prop. 4.4
there), Proposition~\ref{TCE},  and
the easy Lemma~\ref{funny}  but
do not require any information from  Sections~\ref{three},
~\ref{fourr} or ~\ref{four} of the present paper.

\begin{proof}[Proof of Corollary~\ref{cor1}]
Putting together  ~(\ref{derr})
and (\ref{ST1}) we get
\begin{align*}
\Psi_1+
\int \alpha \varphi (\rho_{sal} )'&=
\int  \alpha \varphi' \rho_0 \, dx +\int \alpha \varphi (\rho_{sal})' \\
&=
\int (\id -\LL_1)^{-1} (X' \rho_{sal} + (X\rho_{reg})')\varphi \,  dx\, .
\end{align*}
And, since  the boundary
term in the integration by parts vanishes,
\begin{align*}
\int  \alpha \varphi' \rho_0\, dx +\int \alpha\varphi (\rho _{sal})' &=
\int \alpha \varphi ( -\rho'_0 +(\rho_{sal})')
-\int \alpha' \varphi \rho_0 \, dx \\
&=
-\int \alpha \varphi (\rho_{reg})' \, dx-\int \alpha' \varphi \rho_0 \, dx\, .
\end{align*}
\end{proof}

\begin{proof}[Proof of Theorem~\ref{elegant}]
We suppose that
$c$ is neither periodic nor preperiodic (the other cases are easier).
Recall that $\alpha$ is continuous by Corollary~ \ref{horiz}.
Lemma~ \ref{reallynice} allows us to write $\Psi_1$
as
\begin{equation}\label{first}
\Psi_1 =-\int \varphi \beta'\, ,
\end{equation}
where $\beta'$ is a Stieltjes measure. In fact,
$$
\beta'= \alpha( \rho_{sal})' + (\id -\LL_1)^{-1} (X' \rho_{sal} + (X \rho_{reg})') \, dx \, .
$$
The above implies that $\beta'$ is the sum of an absolutely continuous
measure with density  of bounded variation, and a weighted sum of
diracs along the postcritical orbit.
Now by \cite[Lemma 4.1]{Ba}, we know that $(\id-f_*)(\alpha \rho'_{sal})=X \rho'_{sal}$.
Thus 
\begin{equation}\label{second}
(\id-f_*) (\beta')=X (\rho_{sal} )'+ X' \rho_{sal} + (X\rho_{reg})'= (X \rho_0)'\, .
\end{equation}
Integrating (\ref{first}) by parts, we get
(there are no boundary terms, see e.g. \cite[Proof of Prop. 4.4, Theorem 5.1]{Ba}),
$$
\Psi_1=
\int \varphi'(x) B(x)\, dx \, ,
$$
where $B$ is a function of bounded variation, supported in 
$[a, b]$, satisfying $B'=\beta'$. 
In particular, $B$ is the sum of an element $B_1$ of $BV^{(1)}$ 
with a function with prescribed jumps along the postcritical
orbit. It is easy to check that this
function is in fact just the saltus of $\alpha \rho_{sal}$
(or, equivalently, the saltus of $\alpha\rho_0$).
By
(\ref{second}) (and the fact that both $B(x)$ and
$\rho_0(x)$ vanish for $x \ge b$) we get that
\begin{equation}\label{AA}
(\id -\LL_0)  B= X \rho_0 \, .
\end{equation}
Now, Lemma~\ref{funny} implies that
\begin{equation}\label{BB}
(\id -\LL_0) (\alpha\rho_0)=X\rho_0\, .
\end{equation} 
Putting together (\ref{AA}--\ref{BB}) and $B=B_1+(\alpha\rho_0)_{sal}$, we get that
\begin{equation}\label{CC}
(\id -\LL_0) (B_1-(\alpha \rho_0)_{reg})=0 \, .
\end{equation}
After these preliminaries, we move on to the proof.

If $\alpha$ is absolutely continuous then $(\alpha \rho_0)_{reg}$ is
absolutely continuous (because $\alpha\in BV\cap C^ 0$ and
$((\alpha\rho_0)_{reg})'=\alpha'\rho_0+\alpha(\rho_{reg})'$
is in $L^1(Leb)$).  $B_1$ is absolutely continuous because it is
in $BV^{(1)}$.
The operator $\LL_1$ acting on $L^1(Leb)$ has $\rho_0$ as unique fixed point,
and thus $\LL_0$ on the Banach space of absolutely continuous functions
supported in $(-\infty, b]$ has $R_0(x)=-1+\int_{-\infty}^x \rho_0(y)\, dy$ as unique fixed point.
Thus (\ref{CC}) implies that $B_1=(\alpha \rho_0)_{reg} + \kappa R_0$, so that
$B=\alpha \rho_0+ \kappa R_0$. Since $B(x)=\alpha (x)\rho_0(x)=0$ for
$x \le a$ (use that $\int(X'\rho_{sal} + (X\rho_{reg})') dx=0$ by $\JJ(f,X)=0$),
we have that $\kappa=0$, proving (\ref{derr}).

We next prove the converse. If (\ref{derr}) holds then $B=\alpha \rho_0=
\alpha \rho_{sal}+ \alpha\rho_{reg}$ is in $BV$ by the preliminary remarks. 
Since $\rho_0$ is bounded from below on $[c_2, c_1]$, this implies
that $\alpha|_{[c_2,c_1]}$ is in $BV$. The preliminaries also give 
$B-(\alpha \rho_0)_{sal}=(\alpha \rho_0)_{reg} \in BV^{(1)}$,
i.e., $\alpha' \rho_0+\alpha (\rho_{reg})' \in BV$, which implies that
$\alpha'\rho_0 \in BV$ (since $\alpha\in BV$). Using again $\inf_{[c_2,c_1]}\rho_0 >0$
we get that $\alpha' \in BV$, i.e., $\alpha \in BV^{(1)}$.
\end{proof}


\section{Necessity of the horizontality condition}

There exist  examples of perturbations $f_t$ of  good mixing piecewise expanding 
$C^\infty$ unimodal
maps $f$ with $c$ preperiodic,
$v=X \circ f$ and  $\JJ(f,X)\ne 0$ so that $\RR(t)$ is not Lipschitz 
for some $\varphi \in C^\infty([a,b])$
(\cite[\S 6]{Ba} and  \cite{MM},
see also \cite[Remark 6.3]{Ba}). 
Theorem~\ref{nonlip} below  shows the lack of Lipschitz regularity of $\RR(t)$ for all
perturbations $f_t$ so that the infinitesimal deformation is
not horizontal (we require that $c$ be nonperiodic and, if
$c$ recurrs to itself, $ f'(c_-)=-f'(c_+)$).  The proof of Theorem~\ref{nonlip} hinges on 
a careful rereading of the proof
of   Theorem ~\ref{formula}.

\begin{theorem}\label{nonlip}
Let $f_t$ be a $C^ {2,3}$ perturbation of a mixing piecewise expanding $C^ 3$ unimodal
map $f$ with infinitesimal deformation $v= X \circ f$ such that
$X \in C^2(f(I))$ but $v$ is not horizontal for $f_0=f$,  and assume that $c$
is not periodic for $f$.
If 
$$\gamma=\inf d(f^j(c), c)=0,
$$ 
we assume in addition
that $\lim_{x \to c, x < c} f'(x)=-\lim_{x \to c , x > c} f'(x)$.

If the postcritical orbit of $f_0$  is not dense
in $[c_2,c_1]$ then  there exist $\varphi\in C^\infty(I)$ and $K>0$,  
so that, for any sequence $t_n \to 0$
so that the postcritcal orbit of each $f_{t_n}$ is infinite, there is $n_0 \ge 1$ so that
$$
| \int \varphi \, \rho_{t_n} \, dx-\int \varphi \, \rho_0\, dx|
\ge K |t_n| | \ln |t_n| \, , \quad \forall n \ge n_0\,  .
$$
If the postcritical orbit of $f_0$ is infinite but not dense, the above holds 
for any sequence $t_n \to 0$
with $c$ not periodic under $f_{t_n}$.

If the postcritical orbit of $f_0$ is dense in $[c_2, c_1]$ then there exists
$\varphi\in C^\infty(I)$ so that  for any sequence $t_n \to 0$
so that $c$ not periodic under $f_{t_n}$, we have
 $$
 \lim_{n\to \infty}|t_n^{-1}(\int \varphi \, \rho_{t_n}\, dx-\int \varphi \, \rho_{0}\, dx)|\to \infty\, .
 $$
\end{theorem}

We expect that if $c$ is periodic, but $f=f_0$ is good and  
$\lim_{x \to c, x < c} f'(x)=-\lim_{x \to c , x > c} f'(x)$, 
$v$ is not horizontal, then there exists $\varphi\in C^\infty(I)$ so that the function
$\int \varphi \rho_t\, dx$ is not Lipschitz at $t=0$. 

Existence of sequences $t_n$ as in Theorem~\ref{nonlip} is guaranteed by the
following easy lemma:

\begin{lemma}\label{4.1}
Let $f_t$ be a $C^ {2,2}$ perturbation of a mixing piecewise expanding $C^ 2$ unimodal
map $f$ with infinitesimal deformation $v$.
If $v$ is not horizontal for $f_0$ then
there is a sequence $t_n \to 0$
so that $c$ has an infinite forward orbit for each $f_{t_n}$. 
\end{lemma}

\begin{proof}[Proof of Lemma~\ref{4.1}]
First note that the assumption that $v$
is not horizontal implies that there exists $k_0\ge 1$ so that
$\partial_t c_{k_0,t}|_{t=0}\ne 0$.
Indeed,  assume for a contradiction that
$\partial_t c_{k,t}|_{t=0}= 0$ for all $k\ge 1$.
Then $\partial_t c_{1,t}|_{t=0}= 0$ implies $v(c)=0$, and, using
$v(c_{k})=\partial_t c_{k+1,t}|_{t=0}-f'(c_{k}) v(c_{k-1})$ for $k\ge 1$,
we  prove inductively that $v(c_k)=0$ for all $k\ge 1$, which would
imply that $v$ is horizontal, a contradiction.

Let $\Sigma(t)$ be the symbolic critical itinerary for $f_t$, that is,
$(\Sigma_1(t),\Sigma_2(t),\dots)  \in \{L,C,R\}^{\mathbb{N}}
$,
with  $\Sigma_j(t)=L$ if $f_t^j(c) < c$, $\Sigma_j(t)=C$ if $f_t^j(c)=c$, 
and  $\Sigma_j(t)=R$ if $f^j_t(c) >  c$. Put $\Theta(\Sigma,k_0)=\cap_{n\ge k_0} (f^n_t)^{-1} (I_{\Sigma_n})$, with $I_L=[a,c)$, $I_R=(c,b]$,
$I_C=\{c\}$.
The  map $t \mapsto \Theta(\Sigma(t),k_0)$  is continuous from $(-\epsilon,\epsilon)$
to $\real$. It is easy to see that $\Theta(\Sigma(t),k_0)=c_{k_0,t}$,
so that $\Theta(\Sigma(t),k_0)$  is not constant, and
that is enough to end the proof.
\end{proof}

\begin{proof}[Proof of Theorem~\ref{nonlip}]
The key property that we shall use is
that, for each fixed $k\ge 1$, the limit 
\begin{equation}\label{beta}
\beta_k:=\lim_{t \to 0} \frac{c_{k,t}-c_k}{t} 
\end{equation}
exists and satisfies the twisted cohomological equation
\begin{equation}\label{twisted}
X(c_{k+1})=\beta_{k+1}-f'(c_k) \beta_k \, .
\end{equation}
By definition, $\beta_1=X(c_1)$, so that 
\begin{equation}\label{formulabeta}
\beta_k=\sum_{j=0}^ {k-1} X(c_{k-j}) (f^ {j})'(c_{k-j})  .
\end{equation}
In particular, if $\JJ(f,X)\ne 0$, i.e. if $\alpha_{(0)}(c_1)\ne X(c_1)$ (recall
Lemma~\ref{nice}),
we have   $\beta_1\ne \alpha_{(0)}(c_1)$. 
We shall next have to be a little more careful about the limiting process (\ref{beta}), and distinguish
between the cases where $\gamma$ is zero or strictly positive.

Note that 
$$\beta_k \le |(f^{k-1})'(c_1)| \sup |X| (1-\lambda)^{-1}\, ,
$$ 
(recall (\ref{llambda}) for the definition of $\lambda$)
and put
$$
Y:= \max \{ \sup_t \left |\frac{f_t -f_0}{t}-v\right |_{L^\infty}, \sup_{x \ne c} |f''(x)|, 
\frac{\sup |X|}{1- \lambda^{-1}}, 1 \} \, .
$$
If $\gamma >0$, 
for fixed $t$, we let $M(t) \in \integer$ be the largest integer so that
\begin{equation}
\label{defm}
6 Y^3 |t| |(f^M)'(c_1)| < \gamma/2 \, .
\end{equation}
If $M(t) \ge 1$, it is not  difficult to show inductively that
for all $k \le M(t)$ we have $d(c_{k,t}, c_k) < \gamma/2$ and
\begin{equation}\label{control}
\left | \frac{c_{k,t}- c_k} {t} - \beta_k \right |
\le |t|6 Y^2 | (f^k)' (c_1)| \, .
\end{equation}
Indeed, define $B_{k,t}$ by 
$$
tB_{k,t} (f^k)'(c_1)= (c_{k,t}- c_k) /t - \beta_k \, ,
$$ 
and use
(\ref{twisted}) to see that
\begin{align*}
 \frac{c_{k+1,t}- c_{k+1}} {t} - \beta_{k+1}
&=v'(\tilde w_{k,t})(c_{k,t}-c_k) + t g_t(c_{k,t}) + t f'(c_k) B_{k,t}\\
&\qquad\qquad +t f''(w_{k,t}) (\beta_k +t B_{k, t})^2 \, ,
\end{align*}
where $g_t=(f_t-f_0)/t-v$ and
$w_{k,t}$ and $\tilde w_{k,t}$ are between $c_k$ and
$c_{k,t}$. Then it is easy to see that
$\sup_{k,t} |B_{k,t}| \le 6 Y^2$ for $k \le M(t)$ if $M(t)\ge 1$.

If $\gamma=0$, 
we let $M(t) \in \integer$ be the largest integer so that
\begin{equation}
\label{defmz}
6 Y^3 |t| |(f^M)'(c_1)|  < 1 \, .
\end{equation}
If $M(t) \ge 1$, our assumption that $\lim_{x \to c, x < c}f'(x)=\lim_{x \to c, x > c}f'(x)$
implies that (\ref{control}) holds
for all $k \le M(t)$.

We next revisit the construction  from
Subsection ~\ref{3.2} in order to allow comparison between
different nonperiodic dynamics.
For $\eta >0$,
consider the Banach space $(\widehat \BB_\infty, \|\cdot\|)$ of pairs $\phi=(\phi_{reg}, \phi_{sal})$
with $\phi_{reg}$ continuous and of bounded variation, 
and $\phi_{sal}=(u_k)_{k =1, \ldots, \infty}$,
normed by
\begin{equation}\label{clear2}
\|\phi\|= \|\phi_{reg}\|_{BV} + | \phi_{sal} |_\eta
\mbox{ with }|\phi_{sal}|_\eta =\sup_{1\le k\le\infty}  (1+\eta)^k  |u_k|\, ,
\end{equation}
and so that, in addition,
$
\phi_{reg}(x)=\sum_{k= 1}^{\infty} u_k$
for all $x < a$.
Recall the space $\widehat \BB_t$ associated to $f_t$ in Subsection ~\ref{3.2}.
If the postcritical orbit of $f_t$ is infinite then
$\widehat \BB_t=\widehat \BB_\infty$, and we set $\EE_t=\FF_t$ to be the identity
on $\widehat \BB_\infty$.
If the orbit of 
$c$ is finite (but not periodic) for $f_t$,
letting $n_{0,t}$ and $n_{1,t}$ be minimal so that
$c_{n_{0,t},t}$  is periodic of prime period $n_{1,t}$,
we introduce $\EE_t :\widehat \BB_t \to \widehat \BB_\infty$, 
which maps a finite vector $(w_j, 1\le j \le n_{0,t}+n_{1,t}-1)$ to
an infinite vector $v_\ell$ according to
\begin{align*}
v_\ell&= w_\ell\, , \qquad \ell \le n_{0,t}-1\, , \\
v_{n_{0,t}+j+\ell n_{1,t}}&=w_{n_{0,t}+j}\bigl( (f^{n_{1,t}}_t)'(c_{n_{0,t}+j,t})\bigr)^{\ell}
(1-\bigl ((f^{n_{1,t}}_t)'(c_{n_{0,t}+j,t})\bigr)^{-1})\, ,\\
&\qquad\qquad 0\le j \le n_{1,t}-1\, , \ell \ge 0  \, ,
\end{align*}
and $\FF_t:\widehat \BB_\infty \to \widehat \BB_t$  defined by
\begin{align*}
w_\ell&= v_\ell\, , \qquad \ell \le n_{0,t}-1\, , \\
w_{n_{0,t}+j}&=\sum_{\ell \ge 0} v_{n_{0,t}+j+\ell n_{1,t}}\, ,
&\qquad 0\le j \le n_{1,t}-1  \, .
\end{align*}
It is not difficult to see that $\EE_t$  and $\FF_t$
are bounded,  uniformly in small $t$, and that $\FF_t\circ \EE_t$ is the identity
on $\widehat \BB_t$.

This ends the preliminaries, and we now move on to the proof,
considering $\varphi \in C^\infty(I)$
so that $\int \varphi \, d\rho_0=0$ (this does not restrict
generality).
\end{proof}

\begin{proof}[Proof if the orbit of $c$ is infinite but not dense]
Assume that the closure of $\{ f^j(c) \mid j \ge 0\}$ is an infinite
set which does not coincide with
$[c_2, c_1]$.
Since the orbit of $c$
is not dense in $[c_2, c_1]$, there exists
a $C^\infty$ function $\varphi$ with $\int \varphi \, d \mu_0=0$
and $\varphi(c_j)=1$ for all $j \ge 1$.

Since $\JJ(f,X)\ne 0$,  Lemma~\ref{4.1}
 gives a sequence $t_n \to 0$
so that $c$ is not periodic for $f_{t_n}$. 
For $t=0$ or $t=t_n$ for some $n$, 
put
\begin{align}\label{newG}
\GG_t=\Gamma_{0} \circ  \FF_0\circ  \EE_t \circ \Gamma_t^{-1}:\BB_t \to
\BB_0
\, ,
\widetilde \GG_t=\Gamma_{t} \circ  \FF_t\circ  \EE_0 \circ \Gamma_0^{-1}:\BB_0 \to
\BB_t
\, ,
\end{align} 
(the above operators are bounded uniformly in $t$)
and redefine $\PP_t$ as 
$$\PP_t=\GG_t \circ \LL_{1,t} \circ \widetilde \GG_t:\BB_0\to \BB_0\, .
$$
Since $\EE_0=\FF_0$ is the identity, we find $\widetilde \GG_t\circ  \GG_t=\id$,
and  the spectral decomposition
$\LL_{1,t}^k(\varphi)= \rho_t \int \varphi\, dx + \RR^k_t(\varphi)$, with
$\|\RR^k_t\|_{\BB_t}\le C \tau ^k$, gives a spectral decomposition
$$
\PP_{t}^k(\phi)= \GG_t(\rho_t )\int \widetilde \GG_t( \phi)\, dx + \GG_t(\RR^k_t(\widetilde \GG_t(\phi))) \, .
$$
Using this new definition of
$\PP_t$, we revisit the proof of Theorem ~\ref{formula}, and we study
\begin{equation}\label{easyn'}
\rho_{tn}-\rho_0= (\GG_{t_n} (\rho_{t_n})-\rho_0) + (\rho_{t_n}-\GG_{t_n}(\rho_{t_n}))\,  .
\end{equation}

Assume first that $\gamma >0$.

Let us consider the first term in the right hand side of (\ref{easyn'}).
Step~ 1 of the proof of Theorem \ref{formula} until (\ref{lllast})
uses the fact that $f_t$ and $f_0$ are
conjugate  only (but essentially) to evaluate the
second term of (\ref{easyn'}).
Step ~3 of the proof of Theorem ~\ref{formula}  does not use the fact that 
$f_0$ and $f_t$ are conjugate,  so that 
(\ref{missarg'}) and  (\ref{missarg})
hold.
Step~ 2 of the proof of  Theorem~\ref{formula} appears to use the conjugacies $h_t$,
but a careful look reveals that
what is crucial there are properties (\ref{control})
and (\ref{twisted}) of $\beta_k$.
More precisely, taking $M(t_n)$
from (\ref{defm}) and replacing in Step ~2 the number $\alpha(c_k)$ by $\beta_k$, 
we use
$$
\frac{1}{|(f^{M+1})'(c_1)|}< \frac{12 Y^3 |t_n|}{\gamma}
$$  
to handle the truncated terms for $\ell > M(t_n)$, and deduce that there
is $ C$ depending only on $f$ and on $X$ so that 
$$
\|\PP_{t_n}(\rho_0)- \rho_0\|_{\BB_0} \le C |t_n|, \forall |t_n|< \delta,
t_n\mbox{ not periodic. } 
$$
(Note that $C=O(\gamma^{-1})$.) The above considerations imply that
there is $\widetilde C=O(\gamma^{-1})$ and $\delta > 0$ so that for all $|t_n|< \delta$ with
$c$ not periodic
$$
|\GG_{t_n}(\rho_{t_n})-\rho_0|_{Radon}\le \widetilde C |t_n| \, .
$$
Note that $\delta$ depends only on the constants in the Lasota-Yorke
inequality, on $\lambda$, and on the spectral gap $\tau<1$ of the transfer operator.

We now consider the first term in (\ref{easyn'}), that is
$$
\sum_{k \ge 1} \frac{s_{1,t_n}}{(f^{k-1}_{t_n})'(c_1)}
(H_{c_{k,t_n}}- H_{c_k}) \, .
$$
The terms for $k > M(t_n)$ give a contribution which is
$\le \bar C |t_n|$ for $\bar C=O(\gamma^{-1})$, so that we may restrict to
$k\le M(t_n)$.

Then for $k \ge 1$
$$
\lim_{t_n \to 0} \frac{s_{1,t_n}}{(f^{k-1}_t)'(c_1)}
\int \varphi \frac{H_{c_{k,t}}- H_{c_k}}{t_n} dx = 0 \, .
$$
It is easy to see that
there exists $N=N(f)$ so that $| s_1 \sum_{j=1}^k \frac{X(c_j)}{(f^{j-1})'(c_1)}|\ge \JJ(f,X)/2$
for all $k \ge N$.
Note 
that $N$ depends only on $\lambda$ and $\sup |X|$.
The properties of $\beta_k$ give $|C_{k,n}|\le \widehat C$
and  $|C_n |\le \widehat C$,
uniformly in $n$ and $k$, so that for all $t_n$
small enough so that $M(t_n) >  N$
\begin{align}
\nonumber& \left |\sum_{ k \le M(t_n)} \frac{s_{1,t_n}}{(f^{k-1}_t)'(c_1)}
\int \varphi \frac{H_{c_{k,t_n}}- H_{c_k}}{t_n} dx\right | 
=\left |\sum_{ k \le M(t_n)} (\beta_k + C_{k,n})
 \frac{s_1}{(f^{k-1})'(c_1)}
 \varphi(c_{k}) \right |\\
\nonumber&\qquad\qquad\qquad\qquad\qquad\qquad\qquad\qquad=
\left |C_n + s_1\sum_{ k \le M(t_n)} \varphi(c_k)\sum_{j=1}^k \frac{X(c_j)}{(f^{j-1})'(c_1)}\right |\\
 \label{main}&\qquad\qquad\qquad\qquad\qquad\qquad\qquad\qquad \ge 
(M(t_n)-N) | \frac{|\JJ(f,X)|}{2}- \widehat C \, .
\end{align}
Since $M(t_n)=\Theta(\ln(t_n))$, we have proved the theorem in the
case where the postcritical orbit is infinite but $\gamma >0$.

If the postcritical orbit is infinite and not dense, but
$\gamma=0$ then we should use
definition (\ref{defmz}) for $M(t)$. 
(We still have  $M(t_n)=\Theta(\ln(t_n))$.)
Our additional assumption then yields
constants $\widetilde C$ and $\bar C$  independent
of $\gamma$. 
\end{proof}

\begin{proof}[Proof if $c$ is preperiodic for $f$]
Assume that
$c_{n_0}$ ($n_0 > 0$ minimal for this property, note
that then $n_0 \ge 2$) is periodic of prime period $n_1\ge 1$, 
in particular  $\gamma > 0$.
Take a $C^\infty$ observable with $\int \varphi \, d\mu_0=0$ and
$$
\varphi(c_j)=1,\forall j \ge 1\,  .
$$
By Lemma~\ref{4.1}, there is a sequence $t_n\to 0$
so that $c$ has an infinite forward orbit under $f_{t_n}$.
For $t=0$ or $t=t_n$, recalling (\ref{newG}), consider
$\MM_t=\widetilde \GG_t \circ \LL_{1,0} \circ  \GG_t$ acting on $\BB_t$.
Since $\GG_t \circ \widetilde \GG_t=\id$, we have the spectral decomposition
$$
 \MM^k_t = \widetilde \GG_t(\rho_0) \int \GG_t(\varphi)\, dx+
 \widetilde \GG_t (\RR^k_0(\GG_t(\varphi))\, .
$$
We consider
\begin{equation}\label{easyn''}
\rho_{tn}-\rho_0= (\rho_{t_n})-\widetilde \GG_t (\rho_0)) +
(\widetilde \GG_t(\rho_{0})-\rho_0)\,  .
\end{equation}
Revisiting the proof of Theorem~\ref{formula} once more, using 
$\BB_t$  instead of $\BB_0$, we can treat this case in a manner analogous
to that of the infinite  postcritical orbit with $\gamma >0$.
\end{proof}

\begin{proof}[Proof if the orbit of $c$ is  dense]
We have  $\EE_0=\FF_0=\id$ and, using (\ref{newG}),
we can consider $\PP_t$ as in the case when the orbit is infinite
but not dense. The new difficulty resides in the choice of the
observable.

We recall (\cite[Thm 8.1]{Br}) the following
central limit theorem with speed for $f$ and $\mu_0$.
If $\int \varphi d\mu_0=0$
and if there is no $\tilde \varphi \in BV$ so that $\varphi=\tilde \varphi - \tilde \varphi\circ f$ in $BV$, i.e., except
on an at most countable set
(it is not difficult to see that such $\varphi \in C^\infty(I)$ exist)
then 
$$
\sigma^2:= \lim_{n\to \infty}\int  \left(\frac{\sum_{k=0}^{n-1} \varphi(f^k(x))}{\sqrt n}\right)^2\, d\mu_0> 0\, ,
$$
and there exists  $C(\varphi)$ depending only on the $C^1$-norm of $\varphi$
so that for any $y \in \real$
$$
|\mathbb P (\{x \mid
\sum_{k=0}^{n-1} \varphi(f^k(x)) \le y \sigma \sqrt n\})
-\frac{1}{\sqrt {2\pi}} \int_{-\infty}^y e^{-s^2/2} ds |\le \frac{C(\varphi) }{\sqrt n} \, .
$$
where $\mathbb  P (E)=\int \chi_E d\mu_0$.
Fix $\varphi$ satisfying the above conditions,
$y < 0$ small and let $N_1=N_1(y)$ be so that 
$\frac{C(\varphi) }{\sqrt N_1} < \frac{1}{\sqrt {2\pi}} \int_{-\infty}^y e^{-s^2/2} ds$.
Then there exists $x_0 \in [c_2,c_1]$ so that 
$$
|\sum_{k=0}^{n-1} \varphi(f^k(x_0))| \ge |y| \sigma \sqrt n, \forall n \ge N_1 \, .
$$
Since the postcritical orbit is dense,  for any $\delta > 0$ there
exists $j_0 \ge 1$ so that $d(c_{j_0}, x_0)< \delta$. Put
$\Lambda_f = \sup |f'|$. If $\delta \Lambda_f^m
\le \delta  |(f^m)' (c_{j_0})|<1/2$
for some $m \ge N_1$
then for all $j_0 \le n \le j_0+m$ we have
\begin{align}\label{gluu}
|\sum_{k=0}^{n-1} \varphi(c_{k+1}))|& \ge |\sum_{k=0}^{n-j_0-2} \varphi(f^k(x_0))|- 
2 \sup |\varphi'| -|\sum_{k=1}^{j_0-1} \varphi(c_{k})|\\
\nonumber 
&\ge |y| \sigma \sqrt {n-j_0}-2 \sup |\varphi'|- |\sum_{k=1}^{j_0-1} \varphi(c_{k})| \, .
\end{align}
Assume now for a contradiction that $|\int \varphi d\mu_t|\le A |t|$  for
some $A < \infty$ and all small enough $t$. 
Let $t_n\to 0$ be a sequence of parameters so that $c$ is not
periodic for $f_{t_n}$ (this exists by Lemma~\ref{4.1}).
Recall the argument in the case when the orbit of $c$
is infinite but not dense.
For $|t_n|<\delta_0$, let 
$\widetilde C$  be the Lipschitz constant corresponding to
the first term of (\ref{easyn'}) and let $\bar C$ be
the Lipschitz constant corresponding to the truncated terms for $k\ge M(t_n)$
(where $M(t_n)$ is defined by (\ref{defmz}))
in the second term of (\ref{easyn'}).

For arbitrarily small $t$,
taking $N$ as in the preperiodic case,
the chain of inequalities (\ref{main}) becomes
\begin{align*}
\nonumber &\left |\sum_{ k =1}^{ M(t_n)} \frac{s_{1,t_n}}{(f^{k-1}_{t_n})'(c_1)}
\int \varphi \frac{H_{c_{k,t_n}}- H_{c_k}}{t_n} dx\right |
=\left |C_{n} + s_1\sum_{ k =1}^{ M(t_n)} \varphi(c_k)\sum_{j=1}^k 
\frac{X(c_j)}{(f^{j-1})'(c_1)}\right |\\
&\qquad\qquad\qquad\ge 
|\sum_{ k=1}^{ M(t_n)} \varphi(c_k)|(M(t_n)-N)  \frac{|\JJ(f,X)|}{2}- \widehat C\,  .
\end{align*}
If 
$$
|\sum_{ k =1}^{ M(t_n)} \varphi(c_k)|(M(t_n)-N)  \frac{|\JJ(f,X)|}{2}- \widehat C -\widetilde C
-\bar C
> A \, ,
$$
we have obtained our contradiction.
Otherwise 
$$
|\sum_{ k=1}^{  M(t_n)} \varphi(c_k)|
\le \frac{(A+\widetilde C + \bar C +\widehat C)}{ (M(t_n)-N) } \frac{2}{|\JJ(f,X)|}\, .
$$
If the above held for all small enough $t$, then we would have proved that there
is $\epsilon >0$ and a constant $D(f_t,\varphi)$  with
$$
|\sum_{ k=1}^{ M(t_n)} \varphi(c_k)|
\le \frac{D}{ M(t_n)}\,  , \quad \forall |t_n|<\epsilon \, .
$$
We shall end the proof by showing that the above estimate gives a contradiction.

Recall that $\varphi$, $\sigma$,
$y<0$ and $N_1(y)$ are  fixed, and that we haven chosen a generic $x_0$ as above.
Take $m\ge N_1$ and let  $\delta>0$ be so that $\delta \Lambda^m_f<1/2$.
Then take $j_0(\delta)\ge 1$ so that $d(c_{j_0}, x_0)<\delta$.
If $j_0$ does not tend to infinity as $m\to \infty$,  then, recalling
(\ref{gluu}),  the following
expression tends to infinity as $m\to \infty$
$$
|\sum_{k=0}^{j_0+m-1} \varphi(c_{k+1})|\ge 
|y| \sigma \sqrt {m}-2 \sup |\varphi'|- |\sum_{k=1}^{j_0-1} \varphi(c_k)|\, ,
$$
and we have obtained a contradiction.
Otherwise, up to taking large enough $m$, there exist $s$ so that
$|M(s)-j_0|\le 1$, and $|t_n|\le |s|$ so that $|M(t_n)-j_0-m|\le 1$.
Then, recalling (\ref{gluu}) 
\begin{align*}\label{gla}
&|\sum_{k=0}^{j_0+m-1} \varphi(c_{k+1})|\ge
|y| \sigma \sqrt {m}-2 \sup |\varphi'|- |\sum_{k=1}^{j_0-1} \varphi(c_k)|
\ge |y| \sigma \sqrt {m}-2 \sup |\varphi'|- \frac{D}{j_0}\\
&|\sum_{k=0}^{j_0+m-1} \varphi(c_{k+1})|\le\frac{D}{M(t_n)}\, .
\end{align*}
The righmost lower bound in the first line clearly diverges as $m\to \infty$, giving the
desired contradiction.
\end{proof}


\begin{appendix}
\section{An auxiliary lemma}

\begin{lemma} \label{conju} 
Let $f$ and $g$ be two piecewise expanding
$C^1$ unimodal maps and assume that $c=0$.
If 
$
\sup_x  \{1/|f'(x)|, \ 1/|g'(x) |  \}  \leq  \theta$
 and 
$\sup_{x} |f(x)-g(x)|\leq \delta 
$,
then for all points $x_f$ and $x_g$ such that 
\begin{equation}
\label{same} f^k(x_f) \cdot g^k(x_g) \geq 0 \, , \forall k \le n\, ,
\end{equation}
we have 
$|x_f -x_g| < \theta^n + \frac{\delta}{1-\theta} 
$.
\end{lemma}

\begin{proof} 
We can extend the inverse branches of $f$ and $g$, denoted 
$\psi^f_\sigma$, $\psi^g_\sigma$,  for $\sigma\in \{+,-\}$, 
to  $C^1$ diffeomorphisms  defined on $f(I)\cup g(I)$,
so that they also have derivatives bounded from above by  $\theta$ and 
$$
\max_{\sigma=+,-} \sup_{y \in f(I)\cup g(I)} |\psi^f_{\sigma}(y)- \psi^g_{\sigma}(y)|< 
\delta \, .
$$
Condition (\ref{same}) implies that there exists a sequence $\sigma_k \in \{+,-\}$, 
$k\leq n$, such that 
$$\psi^f_{\sigma_{1}}\circ \cdots \circ \psi^f_{\sigma_n}(f^n(x_f))=x_f \mbox{ and }  
\psi^g_{\sigma_{1}}\circ \cdots \circ \psi^g_{\sigma_n}(f^n(x_g))=x_g \, .
$$
The lemma then follows from
\begin{align*}
&|f^k(x_f)-g^k(x_g)|=|\psi^f_{\sigma_{k+1}}(f^{k+1}(x_f))-
\psi^g_{\sigma_{k+1}}(g^{k+1}(x_g))|
\\
&\, \, \leq
|\psi^f_{\sigma_{k+1}}(f^{k+1}(x_f))-\psi^f_{\sigma_{k+1}}(g^{k+1}(x_g))|+
|\psi^f_{\sigma_{k+1}}(g^{k+1}(x_g))-\psi^g_{\sigma_{k+1}}(g^{k+1}(x_g))|
\\
&\, \, 
\leq  \theta|f^{k+1}(x_f)-g^{k+1}(x_g)|  +  \delta \, .
\end{align*}
\end{proof}


\section{Keller-Liverani bounds for sequences of weak norms}
\label{newKL}

We explain how (\ref{hand0}) and (\ref{missarg'}) imply that for each $\gamma >0$,
there exist $\epsilon_0 >0$ and  $K\ge 1$ so that
\begin{equation}\label{claim0}
\|(z- \PP_t)^{-1}\|_{\BB_0}
\le K  \, , \quad  \forall |t| < \epsilon_0\, ,
\mbox{ if } |z| \ge \tau\mbox{ and }  |z-1| \ge \gamma \, ,
\end{equation}
by
adapting the proof of \cite[Theorem~1]{KL} of Keller and Liverani.
Since we have
\begin{align*}
 (\id -\PP_t)^{-1}(\id-\widehat \Pi_t)(\varphi)=
 -\frac{1}{2i\pi}
 \oint \frac{1}{z-1} (z- \PP_t)^{-1}(\varphi)\, dz \, , \quad
\forall \varphi \in \BB_0\, ,
\end{align*}
(on any contour $|z-1|=\gamma$ with $\gamma \in (0, 1- \tau)$),
the bound (\ref{claim0}) implies that $\|(\id -\PP_t)^{-1}(\id-\widehat \Pi_t)\|_{\BB_0}$
is bounded uniformly in  $|t|<\epsilon_0$, i.e., (\ref{newKL'}).

\smallskip
Fix $\lambda <\tau <1$ as after (\ref{ttau}).
The first remark is that \cite[Lemma 1]{KL} is replaced by the claim
that there exist $\epsilon_1$, $n_1$ and $C_1$, depending only on $C$
from (\ref{hand0}) and on $\tau$,
so that for any $|z| \ge \tau$, all $\varphi \in \BB_0$, all $|t|\le \epsilon_1$
\begin{equation}\label{lemma1}
\| \varphi\|_{\BB_0} \le C_1 \| \widehat Q_t(z) \varphi \|_{\BB_0} + C_1 | \varphi|_{weak, n_1} \, .
\end{equation}
Now,  the beginning of the proof of
\cite[Theorem~1]{KL} gives that that (\ref{hand0}) and (\ref{missarg'}) imply
for   all $m\ge 0$, $n\ge 0$ and all $|z|\ge \tau$, we have
(see \cite[(12)]{KL})
\begin{align}\label{mn}
|\widehat \QQ_t(z)^{-1} \varphi|_{weak, m}
\le &
\biggl (\| \widehat \QQ_0^{-1}(z)\|_{\BB_0}  C(2C +|z|) \left (\frac{\lambda}{\tau}\right )^n\\
\nonumber &
+ \bigl (\| \widehat \QQ_0^{-1}(z)\|_{\BB_0} C + \frac{C}{1-\tau}\bigr ) 
(C\delta_{m+n}(t)) \left (\frac{1}{\tau}\right )^n
 \biggr ) \| \varphi\|_{\BB_0}\\
\nonumber &
+ \bigl (\| \widehat \QQ_0^{-1}(z)\|_{\BB_0} C + \frac{C}{1-\tau}\bigr ) 
 \left (\frac{1}{\tau}\right )^n |\varphi|_{weak, m+n}\, .
\end{align}
Fix $\gamma >0$, write 
$H=\sup_{|z|\ge \tau, |z-1|>\gamma}\| \widehat \QQ_0^{-1}(z)\|_{\BB_0}
$,
and take 
$$
n_2 = \left [ \frac{\ln(4C_1 H C (2C+2) )}{\ln(\tau/\lambda)} \right ] \, .
$$
Then two applications of (\ref{lemma1})  as in the proof of \cite[(15)]{KL} 
(taking $m=n_1$, $n=n_2$ in (\ref{mn}))
show that, taking,
$$\epsilon_0
=\sup \left \{ |t|
\mid  \delta_{n_1+n_2}(t) \biggl (HC+\frac{C}{1-\tau} \biggr ) \left( \frac{1}{\tau}
\right)^{n_2} \le\frac{1}{4C_1}  \right \} \, ,
$$ 
we have
\begin{align*}
\|\varphi\|_{\BB_0} &\le 2 C_1 \| \widehat \QQ_t (z)(\varphi)\|_{\BB_0}
+\frac{1}{2 \delta_{n_1+n_2}(\epsilon_2)} | \widehat\QQ_t(z)(\varphi)|_{weak, n_1+n_2}\\
&\le  K \| \widehat \QQ_t (z)(\varphi)\|_{\BB_0} \, ,
\end{align*}
for all $|t|\le  \epsilon_0$, and
any $|z|\in [ \tau, 2]$ with $|z-1|>\gamma$,  proving (\ref{claim0}).
\end{appendix}
\bibliographystyle{amsplain}

\end{document}